\documentclass[reqno,10pt]{amsart}
\numberwithin{equation}{section}
\usepackage{amsmath}
\usepackage{esint} 
\usepackage{amsthm}
\usepackage{epsfig}
\usepackage{psfrag}
\usepackage{graphicx}
\usepackage{graphpap,latexsym,epsf}
\usepackage{color}
\usepackage{amssymb,mathrsfs,enumerate}
\usepackage[colorlinks, citecolor=citegreen, linkcolor=refred]
{hyperref}
\newcommand{\R}{\mathbb{R}}
\newcommand{\RRR}{\mathbb{R}^{3{\times}3}}

\newcommand{\Sph}{\mathbb{S}}

\newcommand{\pa}{\partial}
\newcommand{\na}{\nabla}
\newcommand{\Om}{\Omega}

\newcommand{\dist}{{\rm dist}}

\newcommand{\dd}{{\rm d}}
\newcommand{\SO}{{\rm SO}}
\newcommand{\Sym}{{\rm Sym}}

\newcommand{\rmD}{{\rm D}}
\newcommand{\tr}{{\rm tr}}
\newcommand{\sym}{{\rm sym}\,}
\newcommand{\Ss}{G}

\mathchardef\emptyset="001F

\definecolor{vgreen}{rgb}{0.1,0.5,0.2}
\definecolor{viola}{RGB}{85,26,139}
\definecolor{citegreen}{rgb}{0,0.6,0}
\definecolor{refred}{rgb}{0.8,0,0}

\newtheorem{theorem}{Theorem}[section]
\newtheorem{remark}[theorem]{Remark}

\newtheorem{assumption}[theorem]{Assumption}
\newtheorem{lemma}[theorem]{Lemma}

\newtheorem{notation}[theorem]{Notation}

\begin{document}
\title[Rigorous derivation of active plate models for 
thin sheets of nematic elastomers]
{Rigorous derivation of active plate models for \\
thin sheets of nematic elastomers}

\author[V.~Agostiniani]{Virginia Agostiniani}
\address{SISSA, via Bonomea 265, 34136 Trieste - Italy} 
\email{vagostin@sissa.it}

\author[A.~DeSimone]{Antonio DeSimone}
\address{SISSA, via Bonomea 265, 34136 Trieste - Italy}
\email{desimone@sissa.it}

\begin{abstract} 
In the context of finite elasticity, we propose plate models describing the spontaneous bending of 
nematic elastomer thin films due to variations along the thickness of the nematic order parameters.
Reduced energy functionals are deduced from 
a three-dimensional description of the system
using rigorous dimension-reduction techniques,
based on the theory of $\Gamma$-convergence.
The two-dimensional models are 
nonlinear plate theories in which deviations 
from a characteristic target curvature tensor 
cost elastic energy.
Moreover, the stored energy functional
cannot be minimised to zero, thus
revealing the presence of residual stresses, 
as observed in numerical simulations.
The following three nematic textures are considered:
\emph{splay-bend} and \emph{twisted} orientation
of the nematic director, and uniform director perpendicular to the mid-plane of the film, with variable degree of nematic order
along the thickness.
These three textures realise three very different structural models: one with only one  stable spontaneously bent configuration,  
a bistable one with two oppositely curved configurations of minimal energy, and a shell with zero stiffness to twisting.
\end{abstract}

\maketitle




\section{Introduction}


The interest in designing objects whose shape can be controlled at 
will through the application of external stimuli 
is fuelling a renewed interest in questions 
at the interface between elasticity and geometry.
Which shapes are accessible to 
elastic sheets through the prescription of non-euclidean metrics that model states of pre-stress or pre-stretch induced by phase-transitions, plastic deformations, or growth
 \cite{Kle_Efr_Sha_2007}?
Besides their fundamental
mathematical interest \cite{Ciarlet_book,Lew_Pak2011}, 
these questions are very relevant in biology 
(e.g., in morphogenesis where shape emerges from
growth and remodelling processes) 
and engineering 
(e.g., for motion-planning problems in soft robotics and,
more generally, for the design of bio-inspired structures with programmable shapes).

A general paradigm to generate bending deformations in thin films 
is to induce non-constant strains through the thickness\footnote{Another route, 
which exploits Gauss' Theorema Egregium,
is to induce curved configurations through 
nematic director textures generating  
(spontaneous) strains that are constant along the thickness 
but variable in the in-plane direction in such a way
to be incompatible with having zero Gaussian
curvature 
\cite
{Aha_Sha_Kup,Arroyo2014,Arroyo2012,Bat_Lew_Sch,LucantonioJCP2017,Bat_Mod_War,
Mod_War2015,Mostajeran,Pismen}.}.
These can in turn be triggered by the spontaneous strains associated with a phase transformation.
An example 
is provided by strips of nematic elastomers in which 
specific textures of the director have been imprinted in the material at fabrication.
The process relies on pouring a nematic liquid between two plates which have been
treated to induce a given uniform alignment 
of the director on one of them and a different one
on the other one. 
This induces a non-constant director profile which is then frozen in
the material by the photo-polymerization process that transforms 
a liquid crystal into a nematic elastomer.
When the isotropic-to-nematic phase transformation takes place, 
the spontaneous deformations associated with it induce differential expansions along the film thickness, 
and hence curvature of its mid-surface.
We refer the interested reader to, e.g.,  
\cite{hybrid,Sawa2011,Urayama2013}
for more details about the preparation of such materials, and to  \cite{AgDe2,BarDeS,Ce_DeSim,CoDeDo,DeS99,DeSim_Dolz,DeTe,FukDyn} and the references quoted therein for further information on the mathematical modelling of their interesting behaviour.

Two-dimensional models (plate models) for the bending deformation of thin films 
made of active material have already 
been proposed in the literature. 
They account for bending deformations through a curvature tensor 
(second fundamental form 
of the deformed mid-surface). 
The bending energy penalizes deviations of the curvature from
a characteristic target curvature arising from the 
spontaneous strains 
triggered by a phase transition.
Expressions for these bending energies are
typically \emph{postulated} on the basis of symmetry arguments, or deduced formally from an ansatz on the 
displacement fields (Kirchoff-Love assumption). By contrast, in our approach two-dimensional energy densities and target curvatures are \emph{deduced}
from 3D elasticity, i.e., from those geometric and material parameters that are available to the material scientists synthesizing the material and shaping it into a thin film.

In this paper, 
employing rigorous dimension reduction techniques based on the theory of $\Gamma$-convergence, and 
following \cite{Schmidt2007}, we derive  new models for
the bending behavior of thin films made of nematic elastomers
in the regime of large deformations.
Starting from three-dimensional finite elasticity, and considering 
the limit of a vanishingly small thickness, 
we obtain the following two-dimensional 
reduced ``energy'' functional
\begin{equation}
\label{energia_limite}
\mathscr E^{lim}(y)
\,=\,
\frac{\mu}{12}\int\limits_{\omega}
\left\{\Big|{\rm A}_y(x')\,-\,\overline A\Big|^2
+\gamma
\Big(\tr{\rm A}_y(x')-\tr\overline A\Big)^2\right\}\dd x'
+
\bar e,
\end{equation}
whenever $y$ is an isometry mapping $\omega$ 
(the planar domain representing the reference 
configuration of the mid-surface of the film) into $\R^3$.
Specific expressions of the 2D limit energy in terms of the parameters typically used in the theory of plates (such as the plate bending modulus) are given
given in the right-hand-sides of \eqref{phys_stat_splay} and \eqref{phys_stat_twist}.
In \eqref{energia_limite} above, 
which is an expression of the type proposed in 
\cite{Aha_Sha_Kup,Kle_Efr_Sha_2007,Armon}
to model the shaping of elastic sheets or of
biological tissues,
the symbol $A_y$ denotes the curvature tensor, namely,  the second fundamental
form associated with the deformed configuration $y(\omega)$
and the coefficients $\mu$ and $\gamma$ are positive constants
(material parameters characterising  the three-dimensional stored energy
density of the material).
Moreover, the $2{\times}2$ symmetric matrix $\overline A$ is the target curvature tensor
and $\bar e$ is a nonnegative constant.
The characteristic quantities
$\overline A$ and $\bar e$ are deduced from the three-dimensional model and given by
explicit formulas, 
issuing from the specific variation of the
spontaneous strain along the thickness.
The constant $\bar e$ is irrelevant in the selection of energy minimising or equilibrium shapes $y$. However, a term $\bar e>0$
is typical for those cases in which the spontaneous strains of the 3D model are not kinematically compatible\footnote{To put the incompatible nematic elastomer cases in perspective,
we also analyze a kinematically compatible case where the 
three-dimensional spontaneous strain distribution
along the thickness depends quadratically on the thickness variable.
We show that, as expected, this distribution leads to 
plates with no residual stresses and where the limiting energy 
corresponding to \eqref{energia_limite} attains its minimum value zero. 
}. Thus, just like its parent 3D energy functional, the limit 2D energy \eqref{energia_limite} 
can never be minimised to zero: there  will always be energy trapped in the system, indicating the presence of residual stresses.

Two special geometries of the director field are of particular interest, 
since they have been realized in practice in the laboratory.
In the \emph{splay bend} geometry (SB),
also called \emph{hybrid} in  
\cite{hybrid,Urayama2013},
the explicit formulas we obtain 
for $\overline A$ and $\bar e$ are 
\begin{equation*}
\overline A
\,=\,
\overline A_{SB}(\delta_0)
\,=\,
\frac{12\,\delta_0}{\pi^2}\,
{\rm diag}\big(-1,0\big)
\qquad\mbox{and}\qquad
\bar e
\,=\,
\bar e_{SB}(\delta_0)
\,=\,
\mu\,(1+\gamma)\delta_0^2\left(\frac{\pi^4-12}8\right)|\omega|,
\end{equation*}
where $\delta_0$ is a positive constant, with dimension of inverse length,
which quantifies the variability along the thickness of 
the spontaneous strain 
(see \eqref{bar_c_h}--\eqref{a_h} and the first formula in \eqref{eq:cose}).
We recall that, 
in a three-dimensional film with splay-bend geometry,
the director continuously rotates by $\pi/2$ from planar to vertical alignment (see \eqref{director_splay} and Figure \ref{fig:splay_twist}).
The other geometry we consider is the \emph{twisted} one (T),
see \cite{Sawa2011} and \cite{Urayama2013},
where instead the director (continuously) rotates perpendicularly to the vertical axis
from a typical orientation at the bottom of the film to another
typical orientation at the top of the film 
(see \eqref{director_splay} and Figure \ref{fig:splay_twist}).
In this case, it turns out that
\begin{equation*}
\overline A
\,=\,
\overline A_T(\delta_0)
\,=\,
\frac{12\delta_0}{\pi^2}\,
{\rm diag}\big(-1,1\big)
\qquad\mbox{and}\qquad
\bar e
\,=\,
\bar e_T(\delta_0)
\,=\,
\frac{\mu\,\delta_0^2}{\pi^4}
\left(\frac{\pi^4-4\pi^2-48}2\right)|\omega|.
\end{equation*}
The difference in the formulas for the two cases arises because of a different distribution
of spontaneous strains along the thickness, 
see \eqref{bar_c_h}--\eqref{a_h} and \eqref{director_splay}.

The two geometries of the director field described above lead to plates with two very different structural behaviours.
Both of them arise from kinematically incompatible spontaneous strains,
which generate residual stresses
leading to a strictly positive constant $\bar e$.
They differ in the fact that in the splay-bend geometry, 
the
integral term in \eqref{energia_limite} can be minimised to zero
(by any developable surface $y(\omega)$ whose second fundamental form $A_y$
coincides with $\overline A_{SB}$).
Hence, the target curvature $\overline A_{SB}$ is the curvature 
the plate spontaneously exhibits in the absence of external loads
(spontaneous curvature). 
By contrast, in the twisted case, also
the minimum of the integral term is strictly positive.
In fact, 
there exists no isometry $y$ such that $A_y\equiv\overline A_T$,
because in a developable surface the product of the principal curvatures
must be zero at each point of the surface.
This means that, in fact,  the target curvature $\overline A_T$ is
never observed in the absence of external loads. 

The spontaneous curvature exhibited  in the absence of external loads by a nematic film with  twisted texture  cannot be read off directly from the target curvature, but is has to be computed by minimising 
the integrand  in \eqref{energia_limite}, subject to the isometry constraint. It turns out that this system has \emph{two} distinct configurations of minimal energy, with opposite curvature, hence it is \emph{bistable}. 
By contrast, in the splay-bend case, there is only one stable bent configuration. Motivated by these observations, we also consider the following geometry for the nematic parameters: 
a uniform director orientation perpendicular to the mid-plane of the film, with variable degree of nematic order
along the thickness. Even though this configuration has not yet been realised in the laboratory, it leads to a very interesting mechanical behaviour. Namely, a structure possessing a continuum of spontaneously bent, minimal energy configurations, representing a shell with zero stiffness to twisting.

The rest of the paper is organised as follows.
In Section \ref{sec_2}, the 3D elasticity models are presented and 
a discussion of the kinematic compatibility
of the 3D spontaneous strains is provided.
Then, in Section \ref{sec_3}, we present the theoretical basis for our dimension reduction procedure, and the derivation of 
the formulas allowing to deduce the target curvature $\overline A$ and the constant $\bar e$ from  3D elasticity.
This is the content of Theorems \ref{main_thm_SB}, \ref{main_thm_T}
and of  formulas \eqref{phys_stat_splay} and \eqref{phys_stat_twist}. 
As already mentioned, 
we work in the framework of the dimension reduction
approach which traces back to the seminal paper \cite{F_J_M_2002}.
In particular, to obtain our results we use the plate theory
for stressed heterogeneous multilayers developed by Schmidt 
in \cite{Schmidt2007}, 
which has recently motivated some new computational schemes
\cite{Bartels}.
Our models are valid for arbitrarily large elastic deformations. 
A plate model covering the regime of small
deformations has been presented in \cite{He}.

Section \ref{sec_4} is devoted to the physical
interpretation of our results:
We derive
explicit formulas for the deformations
realising the minimal free-energy of
the (reduced) plate models, which represent the configuration the nematic sheets exhibit in the absence of applied loads. 
We show that there is one spontaneously bent configuration in the splay bend case, while there are two distinct ones in the twist case. Thus, twist nematic plates are bistable structures, a fact  that has gone unnoticed until now, and has not yet  been  observed in the laboratory.
Moreover, the behaviour of splay-bend
and twisted nematic elastomer sheets is compared to
the case in which the nematic director field is constant (perpendicular to the mid-surface), 
and the thickness-dependence of the 
spontaneous strain is induced by the variation 
of the degree of 
nematic order along the thickness. 
Although a system like this has not yet been synthesised in the laboratory, we hope that the predictions of our model will motivate researchers to  investigate experimentally
the mechanical response it would produce.
In fact, our prediction is that, in the thin film limit, this texture should produce a plate with soft response to twisting, see Figure \ref{fig:sharon}.  


\section{Splay-bend and twisted nematic elastomers thin sheets}
\label{sec_2}


In this section, we present a three-dimensional model
for a thin sheet of nematic elastomer with splay-bend
and twisted distribution of the director along the thickness.
The kinematic compatibility of the corresponding spontaneous
strains is discussed in Subsection \ref{sec:kin_com},
where the case of strains distributed quadratically 
along the thickness is analyzed as well. 

\subsection{A three-dimensional model}
\label{subsec:three_dim_model}

We consider a thin sheet of nematic elastomer occupying the
reference configuration 
\begin{equation}
\label{nostra_ref_con}
\Om_h=\omega\times(-h/2,h/2), 
\end{equation}
for some $h>0$ small, 
where $\omega$ is a bounded Lipschitz domain of $\R^2$
with sufficiently regular boundary.

\begin{notation}
\upshape
Throughout the paper we will denote by
$\{\mathsf e_1,\mathsf e_2,\mathsf e_3\}$ 
the canonical basis of $\R^3$ and
by $z=(z_1,z_2,z_3)$ 
an arbitrary point in the physical reference 
configuration $\Omega_h$.
The term ``physical'' here and throughout the paper
is used in contrast to the corresponding rescaled quantities
we will introduce later on.
Also, $\SO(3)$ is the set of the $3{\times3}$ 
rotations and 
${\rm I}\in\SO(3)$ the identity matrix,
whereas the symbol ${\rm I}_2$ denotes the identity matrix 
of $\R^{2\times2}$.
\end{notation}

We suppose the sheet to be heterogeneous along the thickness
with associated stored energy density
\[
w^h:(-h/2,h/2)\times\RRR\longrightarrow[0,+\infty].
\] 
More precisely, in the two models we are going to consider,
the $z_3$-dependence of the energy density
is induced via the $z_3$-dependence of
the spontaneous strain distribution. 

If $n\in\R^3$ is a unit vector representing 
the local order of the nematic director, 
the (local) response of the nematic elastomer is 
encoded by a volume preserving spontaneous strain 
(technically, a right Cauchy-Green strain tensor)
given by
\begin{equation}
\label{nem_tens}
L(n)=a^{\frac23}n\otimes n+a^{-\frac13}
({\rm I}-n\otimes n),
\end{equation} 
for some material parameter $a>1$,
which is usually temperature-dependent.
Suppose that the nematic director $n$ varies along the thickness 
according to a given function $z_3\mapsto n^h(z_3)$
and coincides with two given constant directions 
at the top and at the bottom of the sheet:
\begin{equation*}
n^h(-h/2)=n_{\rm b},\qquad\qquad
n^h(h/2)=n_{\rm t},
\qquad\qquad\qquad
\mbox{for every small}\quad h>0,
\end{equation*}
for fixed $n_{\rm b},n_{\rm t}\in\Sph^2$.
The through-the-thickness variation of the nematic director 
translates into a 
variation of the corresponding spontaneous strain
according to \eqref{nem_tens}, namely,
\begin{equation}
\label{bar_c_h}
\bar c_h(z_3)
\,:=\,
L(n^h(z_3))
\,=\,
a_h^{2/3}n^h(z_3)\otimes n^h(z_3)
+a_h^{-1/3}\big({\rm I}-n^h(z_3)\otimes n^h(z_3)\big).
\end{equation}
Notice that, in this expression, we allow the material 
parameter $a$ to be $h$-dependent.
More precisely, from now on we will assume that  
\begin{equation}
\label{a_h}
a_h=1+\alpha_0\frac h{h_0},
\end{equation}
where $\alpha_0$ is a positive dimensionless parameter, while $h_0$
and $h$ have the physical dimension of length.
This assumption is easily understandable
if one thinks that curvature is related to the ratio
between the magnitude of the strain difference along the thickness
and the thickness itself. 
Hence, the linear scaling in $h$ in \eqref{a_h} is needed
in order to obtain finite curvature in the limit $h\to0$.
Observe that $\bar c_h(z_3)$ is positive definite
for every $z_3\in(-h/2,h/2)$ and every $h>0$ sufficiently small.

In the framework of finite elasticity, 
a prototypical energy density 
$w^h:(-h/2,h/2){\times}\RRR\to[0,\infty]$
modelling a nematic elastomer is 
\begin{equation}
\label{prot_energy}
w^h(z_3,F):=
\left\{
\begin{array}{ll}
\displaystyle
\frac{\mu}2\Big[(F^TF)\cdot\bar c_h^{-1}(z_3)-3-2\log(\det F)\Big]
+W_{vol}(\det F) & \quad\mbox{if }\det F>0,\\
+\infty & \quad\mbox{if }\det F\leq0,
\end{array}
\right.
\end{equation}
where $\mu>0$ is a material constant 
(shear modulus) and the function
$W_{vol}:(0,\infty)\to[0,\infty)$ is ${\rm C}^2$ around $1$
and fulfills the conditions:
\begin{equation*}
W_{vol}(t)=0 \iff t=1,
\qquad
W_{vol}(t)\longrightarrow\infty\ \,{\rm as}\ \,t\to 0^+,
\qquad
W_{vol}''(1) >0.
\end{equation*}
It is easy to show (see Remark \ref{rmk_W_0}, \eqref{W_h_SB}, and \eqref{W_h_T})
that $w^h$ is indeed nonnegative and such that
\begin{equation*}
w^h(z_3,F)=0\qquad\qquad\mbox{iff}
\qquad\qquad
F\in{\rm SO}(3)\sqrt{\bar c_h(z_3)}.
\end{equation*}
Expression \eqref{prot_energy} is
a natural generalization, see \cite{Ag_DeS_1}, 
of the classical trace formula for nematic elastomers 
derived by Bladon, Terentjev and Warner \cite{Bla_Ter_War},
in the spirit of Flory's work 
on polymer elasticity \cite{Flo}.
The presence of the purely volumetric term $W_{vol}(\det F)$
guarantees that the Taylor expansion at order two of the density 
results in isotropic elasticity with two independent 
elastic constants (shear modulus and bulk modulus).

If $\{\hat f_h\}_{h>0}$, with $\hat f_h:\Om_h\to\R^3$, 
represents a family of applied loads,
the (physical) stored elastic energy and total
energy of the system associated with
a deformation $v:\Om_h\to\R^3$ are given by
\begin{equation}
\label{phys_quan}
\hat{\mathscr E}^h(v)
\,=\,
\int\limits_{\Om_h}w^h(z_3,\na v(z))\,{\rm d}z,
\qquad\qquad
\hat{\mathscr F}^h(v)
\,=\,
\hat{\mathscr E}^h(v)
-
\int\limits_{\Om_h}\hat f_h\cdot v\,\dd z,
\end{equation}
respectively.

Let us now focus on the
nematic director field in the splay-bend and twisted cases,
which we denote by $n^h_{SB}$ and $n^h_T$, respectively.
We recall that these distributions are solutions to
the problem
\begin{equation*}
\min_{\begin{array}{r}n(-h/2)=n^b\\n(h/2)=n^t\end{array}}
\int\limits_{\omega_h}|\na n|^2{\rm d}z,
\end{equation*}
where in the 
splay-bend case $n^b=\mathsf e_1$ and $n^t=\mathsf e_3$, 
whereas in the 
twisted case $n^b=\mathsf e_1$ and $n^t=\mathsf e_2$.
We have
\begin{equation}
\label{director_splay}
n^h_{SB}(z_3)
\,=\,
\left(
\begin{array}{c}
\cos\big(\frac{\pi}4+\frac\pi2\frac{z_3}h\big)\\
0\\
\sin\big(\frac{\pi}4+\frac\pi2\frac{z_3}h\big)
\end{array}
\right),
\qquad
n^h_T(z_3)
\,=\,
\left(
\begin{array}{c}
\cos\big(\frac{\pi}4+\frac\pi2\frac{z_3}h\big)\\
\sin\big(\frac{\pi}4+\frac\pi2\frac{z_3}h\big)\\
0
\end{array}
\right),
\qquad\quad
z_3\in(-h/2,h/2),
\end{equation}
and we refer the reader to Figure 
\ref{fig:splay_twist}
for a sketch of these
two geometries.

We define the (physical) spontaneous strain distributions
$\bar c_{h,SB}$ and $\bar c_{h,T}$ as that in
\eqref{bar_c_h} with $n^h_{SB}$ and $n^h_T$
in place of $n^h$, respectively.
Correspondingly, we denote by
$w^h_{SB}$ and $w^h_T$ the stored energy densities,
by $\hat{\mathscr E}^h_{SB}$ and $\hat{\mathscr E}^h_T$ 
the stored energy functional,
and by $\hat{\mathscr F}^h_{SB}$ and $\hat{\mathscr F}^h_T$ 
the total energies.

\begin{figure}[htbp]
\begin{center}
\includegraphics[width=13cm]{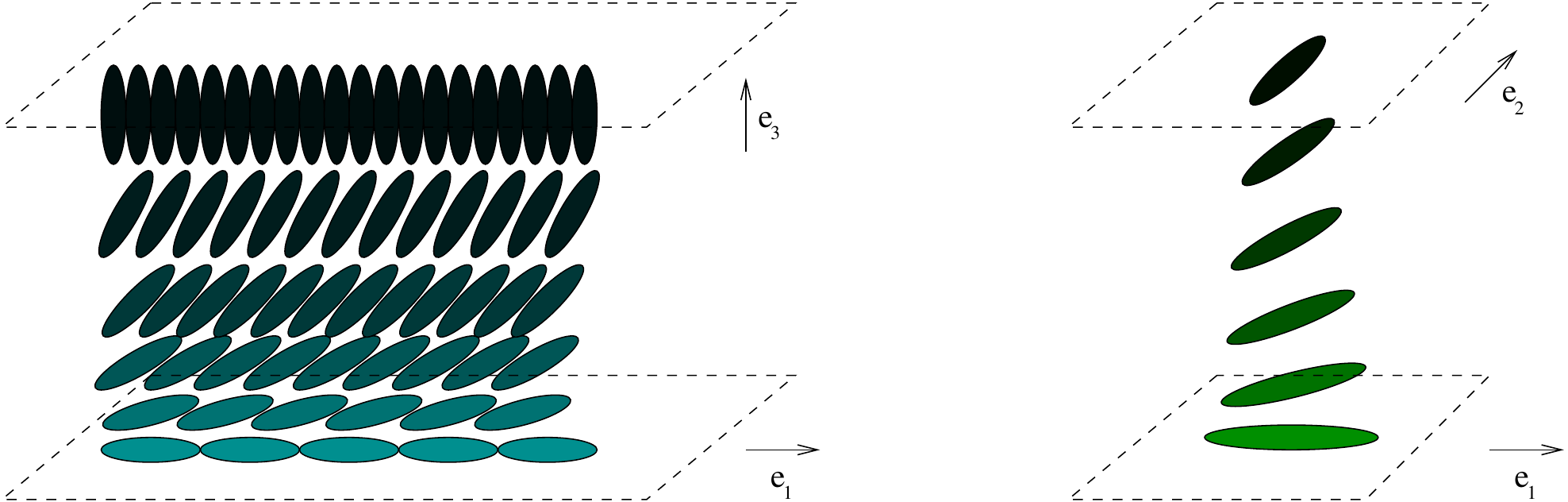}
\end{center}
\caption{
Sketch of the splay-bend director field (left)
and of the twisted director field (right).
} 
\label{fig:splay_twist}
\end{figure}


\subsection{Kinematic compatibility}
\label{sec:kin_com}


Here, we want to discuss the 
\emph{kinematic compatibility} 
of some given field of (physical) spontaneous strains. 
Let $\mathcal O$ be the (physical) reference configuration
of a given system and suppose
that it is a simply connected open subset of $\R^3$. 
We say that a smooth map
$\Ss:\mathcal O\longrightarrow\R^{3{\times}3}$,
representing a distribution of spontaneous strains and 
such that $\Ss(z)\in{\rm Psym}(3)$ for every $z\in\mathcal O$,
is \emph{kinematically compatible} if there exists
a smooth function $v:\mathcal O\longrightarrow\R^3$, 
representing a deformation
and such that $\det\na v(z)>0$ for every $z\in\mathcal O$,
satisfying 
\begin{equation}
\label{kin_com}
\na v^T\na v=\Ss\qquad\mbox{in}\quad\mathcal O.
\end{equation}
Following \cite{Ciarlet_book}, we reformulate this concept
in the framework of Riemannian geometry. 
In order to do this, let us denote by $g_{\R^3}$ the
Euclidean metric of $\R^3$ and recall that
for a given immersion $\varphi:\mathcal O\longrightarrow(\R^3,g_{\R^3})$
the pull-back metric of $g^{\R^3}_{|\varphi(\mathcal O)}$ via
$\varphi$ is the metric $h$ defined in $\mathcal O$ by the identity
\begin{equation*}
h_{|p}(X,Y)
\,=\,
g^{\R^3}_{|\varphi(p)}
\left(\,{\rm d}\varphi_{|p}[X],{\rm d}\varphi_{|p}[Y]\,\right),
\qquad\qquad
\mbox{for every }
\quad
X,Y\in{\rm T}_p\mathcal O=\R^{3{\times}3}.
\end{equation*}
The pull-back metric of $g$ via $\varphi$ is usually denoted by
$\varphi^*\,g_{|\varphi(\mathcal O)}$.
If $(z^i)_{i=1,2,3}$ and $(x^{\alpha})_{\alpha=1,2,3}$
are systems of coordinates for $\mathcal O$ and $\varphi(\mathcal O)$,
respectively, 
the above identity specialized to
$X=\pa/\pa z^i_{|p}$ and $Y=\pa/\pa z^j_{|p}$ gives
\begin{equation*}
{h_{ij}}_{|p}
\,=\,
{g^{\R^3}_{\alpha\beta}}_{|\varphi(p)}
{\frac{\pa\,\varphi^{\alpha}}{\pa z^i}}_{|p}
\frac{\pa\,\varphi^{\beta}}{\pa z^j}_{|p},
\end{equation*} 
where
$\varphi^{\alpha}:=x^{\alpha}\circ\varphi$.
If, in addition, we assume $(x^{\alpha})_{\alpha=1,2,3}$ to be
the standard Euclidean coordinates, 
then the coefficient ${g^{\R^3}_{\alpha\beta}}_{|\varphi(p)}$ 
is just $\delta_{\alpha\beta}$. 
Note that here and in what follows the Einstein summation convention for the sum over repeated indices is adopted.
We identify the given spontaneous strain distribution 
$\Ss$ with a metric defined in $\mathcal O$,
so that asking if there is an
orientation-preserving deformation 
$v:\mathcal O\longrightarrow\R^3$
such that \eqref{kin_com} holds true corresponds to seeking
for a local diffeomorphism
$v:\mathcal O\longrightarrow\R^3$
such that the pull-back metric of $g_{\R^3}$ via $v$ 
coincides with the metric $\Ss$. In formulas,
\begin{equation*}
{\Ss_{ij}}_{|p}
\,=\,
\delta_{\alpha\beta}
{\frac{\pa\,v^{\alpha}}{\pa z^i}}_{|p}
\frac{\pa\,v^{\beta}}{\pa z^j}_{|p},
\end{equation*} 
where we have fixed standard Euclidean coordinates in the target 
manifold $v(\mathcal O)$.
Note that since $v$ is a local differmorphism, 
the equivalence
\begin{equation}
\label{kin_com_2}
v^*g^{\R^3}_{|v(\mathcal O)}=\Ss\qquad\mbox{in}\qquad\mathcal O
\end{equation}
establishes a local isometry (through $v$) between
the manifolds
$\big(\mathcal O,\Ss\big)$ and 
$\Big(v(\mathcal O),g^{\R^3}_{|v(\mathcal O)}\Big)$.
Now, we have from Theorem 1.5-1 and Theorem 1.6-1 
in \cite{Ciarlet_book} that,
since $\mathcal O$ is simply connected,
a necessary and sufficient condition for \eqref{kin_com_2} to hold
is that 
\begin{equation}
\label{kin_com_3}
{\rm Riem}_\Ss\equiv0\qquad\mbox{in}\qquad\mathcal O,
\end{equation}
where ${\rm Riem}_\Ss$ is the fourth-order Riemann curvature tensor
associated with the metric $\Ss$.
We recall that, in the given local chart $(z^i)_{i=1,2,3}$
of $\mathcal O$, the $(3,1)$--coefficients ${\rm R}_{ijk}^l$'s of
${\rm Riem}_{\Ss}
=
{\rm R}_{ijk}^l
\big(
{\rm d}z^i\otimes{\rm d}z^j\otimes{\rm d}z^k
\otimes\partial/\partial z^l
\big)$ are given by
\begin{equation*}
{\rm R}_{ijk}^l
\,:=\,
\frac{\partial}{\partial z^j}
\Gamma_{ik}^l
-\frac{\partial}{\partial z^k}
\Gamma_{ij}^l
+\Gamma_{js}^l\Gamma_{ik}^s
-\Gamma_{ks}^l\Gamma_{ij}^s,
\end{equation*}
where the Christoffel's symbols $\Gamma_{ij}^k$'s are defined as
\begin{equation}
\label{K_symbols}
\Gamma_{ij}^k
\,:=\,
\Ss^{kl}\,\Gamma_{ijl},
\qquad\qquad
\Gamma_{ijl}
\,:=\,
\frac12\Big(\partial_i\Ss_{jl}
+\partial_j\Ss_{il}
-\partial_l\Ss_{ij}
\Big),
\end{equation}
and the symbols $\Ss^{ij}$'s stand for 
the components of the inverse $\Ss^{-1}$ of $\Ss$.
To simplify the computations it is sometimes useful to introduce
the $(4,0)$--coefficients ${\rm R}_{lijk}$'s of ${\rm Riem}_L$,
defined as
\begin{equation*}
{\rm R}_{lijk}
\,:=\,
\Ss_{ls}\,{\rm R}_{ijk}^s
\,=\,
\partial_j\Gamma_{ikl}
-
\partial_k\Gamma_{ijl}
+
\Gamma_{ij}^p\Gamma_{klp}
+
\Gamma_{ik}^p\Gamma_{jlp}.
\end{equation*}
It is clear that ${\rm Riem}_\Ss\equiv0$ if and only if
${\rm R}_{lijk}\equiv0$ for every $l,i,j,k\in\{1,2,3\}$.
Finally, let us recall that, since we are in dimension $3$, 
condition \eqref{kin_com_3} is equivalent to 
${\rm Ric}_{\Ss}\equiv 0$ in $\mathcal O$,
where ${\rm Ric}_{\Ss}$ denotes the second-order 
Ricci curvature tensor associated with $\Ss$, which is defined as
${\rm Ric}_{\Ss}={\rm R}_{ij}\,{\rm d}z^i\otimes{\rm d}z^j$, with
\begin{equation*}
{\rm R}_{ij}\,:=\,\partial_l\Gamma_{ij}^l-\partial j\Gamma_{il}^l
+\Gamma_{lk}^l\Gamma_{ij}^k-\Gamma_{jk}^l\Gamma_{il}^k.
\end{equation*}
From now on in this section, we restrict our attention to the case
where $\mathcal O=\Omega_h$ 
(see \eqref{nostra_ref_con}) and 
the spontaneous strain distribution $\Ss$ is a function of the 
thickness variable $z_3\in(-h/2,h/2)$.
Note that a material point of $\Omega_h$, 
normally referred to as a point of components $(z_1,z_2,z_3)$
throughout the paper, 
is a point of the manifold $\Omega_h$ with coordinates $(z^1,z^2,z^3)$
from the point of view of Riemannian geometry.
In the following subsections, we discuss the 
kinematic compatibility of $z_3\mapsto\Ss(z_3)$ in three cases: 
the case where $\Ss(z_3)$ depends quadratically on $z_3$ and two cases
(splay-bend and twisted nematic elastomer sheets)
where the dependence of $\Ss(z_3)$ on $z_3$ is more complicated
and gives rise to incompatible strains. 
Throughout this section we use the variable $t$ in place
of $z_3$ and we use the index/apex ``$t$'' in place of ``$3$''.


\subsubsection{The splay-bend case} 


In this case, setting 
\begin{equation}
\label{f_h}
f_h(t):=\frac{\pi}4+\frac{\pi}{2h}t,
\qquad\qquad t\in(-h/2,h/2),
\end{equation}
and looking at \eqref{bar_c_h} and \eqref{director_splay},
we have that, up to a multiplicative constant,
the spontaneous strain distribution is given by
\begin{align*}
\Ss =\Ss(t)&=
{\rm I}
+(a_h-1)n^h_{SB}\otimes n^h_{SB}
=
\left(
\begin{array}{ccc}
1+(a_h-1)\cos^2f_h
& 0 & 
\left(\frac{a_h-1}2\right)\sin(2f_h)\\
0 & 1 & 0\\
\left(\frac{a_h-1}2\right)\sin(2f_h)
& 0 &
1+(a_h-1)\cos^2f_h
\end{array}
\right),\\
\Ss^{-1}(t)&=
{\rm I}+\left(\frac1{a_h}-1\right)
n^h_{SB}\otimes n^h_{SB}
=
\left(
\begin{array}{ccc}
1+\left(\frac1{a_h}-1\right)\cos^2f_h
& 0 & 
\left(\frac{1-a_h}{2a_h}\right)\sin(2f_h)\\
0 & 1 & 0\\
\left(\frac{1-a_h}{2a_h}\right)\sin(2f_h)
& 0 &
1+\left(\frac1{a_h}-1\right)\cos^2f_h
\end{array}
\right).
\end{align*}
It turns out that the coefficient ${\rm R}_{1t}$
of ${\rm Ric}_{\Ss}$ has the quite simple expression
\begin{equation}
R_{1t}
\,:=\,
\partial_l\Gamma_{1t}^l-\partial_t\Gamma_{1l}^l
+\Gamma_{lk}^l\Gamma_{1t}^k+\Gamma_{tk}^l\Gamma_{1l}^kì
\,=\,
-\partial_t(\Gamma_{11}^1+\Gamma_{12}^2)
+\Gamma_{lk}^l\Gamma_{1t}^k+\Gamma_{tk}^l\Gamma_{1l}^k
\,=\,
-\partial_t\Gamma_{11}^1
+\Gamma_{1t}^1\Gamma_{1t}^t+\Gamma_{tt}^1\Gamma_{11}^t,
\label{R_1t}
\end{equation}
where we have first used the fact that
the Christoffel symbols depend only on $t$ and 
secondly the property
\begin{equation*}
\Gamma_{ij}^k=0\qquad\mbox{whenever}\qquad
2\in\{i,j,k\}.
\end{equation*}
This can be easily checked using the definition
of $\Gamma_{ij}^k$ in \eqref{K_symbols}.
The same definition and simple computations also give
\begin{align*}
\Gamma_{11}^1
&\,=\,
-\Gamma_{1t}^t
\,=\,
-\frac{(a-1)^2}a\left(\frac{\pi}{2h}\right)
\sin^2f_h\cos^2f_h,\\
\Gamma_{1t}^1
&\,=\,
-(a-1)\left(\frac{\pi}{2h}\right)
\sin f_h\cos f_h
\left[1+\left(\frac 1a-1\right)\cos^2f_h\right],\\
\Gamma_{tt}^1
&\,=\,
(a-1)\left(\frac{\pi}{2h}\right)
(\cos^2f_h-\sin^2f_h)
\left[1+\left(\!\frac 1a-1\!\right)\sin^2f_h\right]
-
\frac{(a-1)^2}a\left(\frac{\pi}{2h}\right)
\sin^2f_h\cos^2f_h,\\
\Gamma_{11}^t
&\,=\,
(a-1)\left(\frac{\pi}{2h}\right)
\sin f_h\cos f_h
\left[1+\left(\!\frac 1a-1\!\right)\sin^2f_h\right]\!.
\end{align*}
Plugging these expressions into \eqref{R_1t} yields
\begin{equation*}
{\rm R}_{1t}
=
-\frac{(a-1)^2}a\left(\frac{\pi}{2h}\right)
\sin f_h\cos f_h
(\cos^2f_h-\sin^2f_h)
=
-\frac{(a-1)^2}a\left(\frac{\pi}{8h}\right)
\sin\left(\pi+\frac{2\pi}ht\right).
\end{equation*}
Thus, we can conclude that
${\rm R}_{1t}$ is not identically zero in $\Om_h$.
In turn, ${\rm Ric}_{\Ss}$ is not
identically zero, so that the splay-bend spontaneous strain
distribution is not kinematically compatible.
 

\subsubsection{The twisted case} 


In this case, 
following the same notation as in \eqref{f_h},
we have
\begin{align*}
\Ss=\Ss(t)&=
{\rm I}
+(a_h-1)n^h_T\otimes n^h_T
=
\left(
\begin{array}{ccc}
1+(a_h-1)\cos^2f_h
& 
\left(\frac{a_h-1}2\right)\sin(2f_h) & 0\\
\left(\frac{a_h-1}2\right)\sin(2f_h) 
&
1+(a_h-1)\cos^2f_h & 0\\
0 & 0 & 1\\
\end{array}
\right),\\
\Ss^{-1}(t)&=
{\rm I}
+\left(\frac1{a_h}-1\right)
n^h_T\otimes n^h_T
=
\left(
\begin{array}{ccc}
1+\left(\frac1{a_h}-1\right)\cos^2f_h 
& 
\left(\frac{1-a_h}{2a_h}\right)\sin(2f_h) & 0\\
\left(\frac{1-a_h}{2a_h}\right)\sin(2f_h)
&
1+\left(\frac1{a_h}-1\right)\cos^2f_h & 0\\
0 & 0 & 1
\end{array}
\right).
\end{align*}
For the twisted geometry, the coefficient 
${\rm R}_{tt}$ of ${\rm Ric}_{\Ss}$ has a simple expression.
Indeed, we have
\begin{equation}
{\rm R}_{tt}
\,:=\,
\partial_l\Gamma_{tt}^l-\partial_t\Gamma_{tl}^l
+\Gamma_{lk}^l\Gamma_{tt}^k+\Gamma_{tk}^l\Gamma_{tl}^k
\,=\,
-\partial_t(\Gamma_{1t}^1+\Gamma_{2t}^2)
-
\big(\Gamma_{tk}^1\Gamma_{1t}^k+\Gamma_{tk}^2\Gamma_{2t}^k\big)
\,=\,
-\Big[
(\Gamma_{1t}^1)^2+2\,\Gamma_{1t}^2\Gamma_{2t}^1
+(\Gamma_{2t}^2)^2
\Big],
\label{R_tt}
\end{equation}
since
\begin{equation}
\label{conti}
\Gamma_{tt}^k=0\quad\mbox{for every}\quad k=1,2,t,
\qquad\mbox{and}\qquad
\Gamma_{1t}^1\,=\,-\Gamma_{2t}^2
\,=\,
-\left(\frac{a^2-1}{2a}\right)\left(\frac{\pi}{2h}\right)
\sin f_h\cos f_h.
\end{equation}
This can be easily checked from the definition of the Christoffel 
symbols in \eqref{K_symbols}. Similar computations
yield
\begin{equation*}
\Gamma_{1t}^2
\,=\,
\left(\frac{a-1}2\right)\left(\frac{\pi}{2h}\right)
\left(\cos^2f_h-\frac{\sin^2f_h}a\right),
\qquad
\Gamma_{2t}^1
\,=\,
\left(\frac{a-1}2\right)\left(\frac{\pi}{2h}\right)
\left(\frac{\cos^2f_h}a-\sin^2f_h\right).
\end{equation*}
Using these formulas together with 
\eqref{R_tt} and the second equation in \eqref{conti} gives
\begin{equation*}
{\rm R}_{tt}
\,=\,
-\frac{(a-1)^2}{2a}\left(\frac{\pi}{2h}\right)^2.
\end{equation*} 
This fact implies, in particular, that
${\rm Ric}_{\Ss}$ is not
identically zero and in turn that the twisted 
spontaneous strain distribution is not kinematically compatible.


\subsubsection{The quadratic case} 
\label{comp_quadrat}


In this subsection, we consider the case where
\begin{equation}
\label{SS_quadratic}
\Ss\,=\,\Ss(t)\,=\,
{\rm I}
+t\,A+t^2\,B,\qquad\qquad t\in(-h/2,h/2),
\end{equation}
for some diagonal matrices $A={\rm diag}(A_{11},A_{22},A_{tt})$
and $B={\rm diag}(B_{11},B_{22},B_{tt})$.
Note that $\Ss(t)=:{\rm diag}(\Ss_{11}(t),\Ss_{22}(t),\Ss_{tt}(t))$
is positive definite for every $t$ sufficiently small.
Elementary computations yield
\begin{align*}
{\rm R}_{11}
&=
-\frac14\left(\frac{\Ss_{11}}{\Ss_{tt}}\right)
\left\{
2\frac{\rm d}{{\rm d}t}\left(\frac{\dot\Ss_{11}}{\Ss_{11}}\right)
+
\left(\frac{\dot\Ss_{11}}{\Ss_{11}}\right)
\left[
\frac{\dot\Ss_{11}}{\Ss_{11}}
+
\frac{\dot\Ss_{22}}{\Ss_{22}}
-
\frac{\dot\Ss_{tt}}{\Ss_{tt}}
\right]
\right\},\\
{\rm R}_{22}
&=
-\frac14\left(\frac{\Ss_{22}}{\Ss_{tt}}\right)
\left\{
2\frac{\rm d}{{\rm d}t}\left(\frac{\dot\Ss_{22}}{\Ss_{22}}\right)
+
\left(\frac{\dot\Ss_{22}}{\Ss_{22}}\right)
\left[
\frac{\dot\Ss_{11}}{\Ss_{11}}
+
\frac{\dot\Ss_{22}}{\Ss_{22}}
-
\frac{\dot\Ss_{tt}}{\Ss_{tt}}
\right]
\right\},\\
{\rm R}_{tt}
&=
-\frac14
\left\{
2\frac{\rm d}{{\rm d}t}
\left(
\frac{\dot\Ss_{11}}{\Ss_{11}}+\frac{\dot\Ss_{22}}{\Ss_{22}}\right)
-
\left(\frac{\dot\Ss_{tt}}{\Ss_{tt}}\right)
\left(
\frac{\dot\Ss_{11}}{\Ss_{11}}
+
\frac{\dot\Ss_{22}}{\Ss_{22}}
\right)
+
\left(
\frac{\dot\Ss_{11}}{\Ss_{11}}
\right)^2
+
\left(
\frac{\dot\Ss_{22}}{\Ss_{22}}
\right)^2
\right\},\\
{\rm R}_{ij}
&=0,
\qquad\mbox{for every }i\neq j,
\end{align*}
where
$\dot\Ss_{ii}$ is the derivative of $\Ss_{ii}$ with respect to $t$.
Now, set
\begin{equation*}
\xi:=\log\Ss_{11},\qquad\quad
\eta:=\log\Ss_{22},\qquad\quad
\tau:=\log\Ss_{tt},
\end{equation*}
so that
$\dot\Ss_{11}/\Ss_{11}=\dot\xi$,
$\dot\Ss_{22}/\Ss_{22}=\dot\eta$,
$\dot\Ss_{tt}/\Ss_{tt}=\dot\tau$,
and in turn
\begin{align*}
{\rm R}_{11}
&\,=\,
-\frac14\left(\frac{{\rm e}^\xi}{{\rm e}^\tau}\right)
\left[
2\,\ddot\xi
+
\dot\xi
\left(
\dot\xi+\dot\eta-\dot\tau
\right)
\right],
\qquad\qquad
{\rm R}_{22}
\,=\,
-\frac14\left(\frac{{\rm e}^\eta}{{\rm e}^\tau}\right)
\left[
2\,\ddot\eta
+
\dot\eta
\left(
\dot\xi+\dot\eta-\dot\tau
\right)
\right],\\
{\rm R}_{tt}
&\,=\,
-\frac14\left[2\,(\ddot\xi+\ddot\eta)
-\dot\tau\,(\dot\xi+\dot\eta)
+
(\dot\xi)^2+(\dot\eta)^2
\right].
\end{align*}

The condition ${\rm Ric}_\Ss\equiv0$, which guarantees the kinematic
compatibility of $t\mapsto\Ss(t)$ as discussed above, 
is then equivalent to the following system of ODEs:
\begin{equation*}
\begin{cases}
&2\,\ddot\xi
+
\dot\xi
\left(
\dot\xi+\dot\eta-\dot\tau
\right)=0,\\
&2\,\ddot\eta
+
\dot\eta
\left(
\dot\xi+\dot\eta-\dot\tau
\right)=0,\\
&2\,(\ddot\xi+\ddot\eta)
-\dot\tau\,(\dot\xi+\dot\eta)
+
(\dot\xi)^2+(\dot\eta)^2=0.
\end{cases}
\end{equation*}
Solving this system translates into compatibility 
conditions on $A$ and $B$ in \eqref{SS_quadratic}. 
It then turns out that a spontaneous strain distribution 
$t\mapsto\Ss(t)$ of the form \eqref{SS_quadratic} 
is kinematically compatible
if and only if one of the following four conditions is satisfied:
\begin{align*}
(i)\quad&
A_{11}=A_{22}=A_{tt}=0\quad\mbox{and}\quad B_{11}=B_{22}=B_{tt}=0,\\
(ii)\quad& 
A_{22}=A_{tt}=0,\quad B_{22}=B_{tt}=0,\quad
\mbox{and}\quad B_{11}=A_{11}^2/4\neq0,\\
(iii)\quad& 
A_{11}=A_{tt}=0,\quad B_{11}=B_{tt}=0,\quad
\mbox{and}\quad B_{22}=A_{22}^2/4\neq0,\\
(iv)\quad& 
A_{11}=A_{22}=0,\quad B_{11}=B_{22}=0,\quad
\mbox{and}\quad A_{tt}^2+B_{tt}^2\neq0.
\end{align*}
Note that the first condition corresponds to the trivial case
$\Ss={\rm I}$ 
and the second one tells us in particular that
a strain of the form
\begin{equation}
\label{SS_specifico}
\Ss(t)
:=
\left(
\begin{array}{ccc}
1-2\,k\,t+k^2\,t^2 & 0 & 0\\
0 & 1 & 0\\
0 & 0 & 1
\end{array}
\right),
\qquad\qquad t\in(-h/2,h/2),
\end{equation}
for some constant $k\in\R\setminus\{0\}$, is kinematically compatible.
A prototypical deformation $v$ giving rise to such $\Ss$ can be
provided in the following way.
Let $I$ be an open interval of $\R$,
let $\gamma:I\longrightarrow\R^3$ be a smooth curve,
and define
\begin{equation*}
T(s):=\gamma'(s),
\qquad\qquad
N(s):=\frac{T'(s)}{|T'(s)|},
\qquad\qquad
B(s):=T(s)\wedge N(s),
\end{equation*}
for every $s\in I$,
where the apex stands for differentiation with respect to $s$. 
Suppose that the curve is parameterized by arc length, so that
$|T|=1$ and the curvature $k$ is defied as $k:=|T'|$.
Then the Frenet--Serret formulas read
\begin{equation*}
\begin{cases}
T'
&=
\ k\,N,\\
N'
&=
\ -k\,T+\tau\,B,\\
B'
&=
\ -\tau\,N.
\end{cases}
\end{equation*}
Note that multiplying the first equation by $N$
gives $k=T'\cdot N=-N'\cdot T$.
Let us restrict to the case of $B$ being 
constantly equal to $\mathsf e_2$,
where $\{\mathsf e_1,\mathsf e_2, \mathsf e_2\}$ is the
canonical basis of $\R^3$.
This means that $\gamma$ is a planar curve
and the above formulas imply in particular that
$\tau=-B'\cdot N=0$ and $|N'|^2=k^2$.
Now, let us define
$v:\Omega_h\longrightarrow\R^3$ as
\begin{equation*}
v(s,z_2,t)
\,:=\,
\gamma(s)+t\,N(s)+z_2\,\mathsf e_2,
\end{equation*}
where we have supposed $\omega_h=\omega\times(-h/2,h/2)$
with $\omega=I\times J$, 
for some open intervals $I$, $J\subset\R$.
Then
$\na v=(T+t\,N'\,|\,\mathsf e_2\,|\, N)$ and therefore
\begin{equation*}
\na v^T\na v
=
\left(
\begin{array}{ccc}
|T|^2+2\,t\,T{\cdot}N'+t^2|N'|^2 
& 
\mathsf e_2{\cdot}(T+t\,N')  
& 
N{\cdot}(T+t\,N')\\
\mathsf e_2{\cdot}(T+t\,N')
& 
|\mathsf e_2|^2 
& 
\mathsf e_2\cdot N\\
N{\cdot}(T+t\,N')
&
\mathsf e_2\cdot N
&
|N|^2
\end{array}
\right)
=
\left(
\begin{array}{ccc}
1-2\,k\,t+k^2\,t^2 & 0 & 0\\
0 & 1 & 0\\
0 & 0 & 1
\end{array} 
\right).
\end{equation*}
Supposing the curvature $k$ to be constant, we have thus derived
a strain $\Ss$ of the form \eqref{SS_specifico}.

Finally, note that the analysis performed in this session
shows that in the linear case where $\Ss(t)$ is of the form
\begin{equation*}
\Ss\,=\,\Ss(t)\,=\,
{\rm I}
+t\,A,\qquad\qquad t\in(-h/2,h/2),
\end{equation*}
for some diagonal matrix $A\neq0$, 
the kinematic compatibility of the spontaneous strain
distribution is never fulfilled.


\section{Derivation of the plate model}
\label{sec_3}


In this section, 
we first rewrite the three-dimensional model previously introduced
in a rescaled reference configuration. 
Then, in Subsection \ref{rigorous_res}, 
we recall two rigorous dimension reduction results
of compactness and $\Gamma$-convergence. 
This mathematical technique is subsequently employed
in Subsection \ref{nostra_derivazione}, 
where our main results, 
Theorems \ref{main_thm_SB} and \ref{main_thm_T},
are stated and proved.


\subsection{The rescaled three-dimensional model}
\label{3D_rescaled}


As it is standard for dimension reduction techniques, 
let us now operate a change of variables in order to
rewrite the energies in a fixed, $h$-independent 
rescaled reference configuration.

\begin{notation}
\upshape
We denote by $x=(x_1,x_2,x_3)=(x',x_3)$ 
an arbitrary point in the rescaled reference configuration 
$\Om:=\omega\times(-1/2,1/2)$.
\end{notation}

For every $h>0$ small, we define the rescaled energy density
$W^h:(-1/2,1/2)\times\R^{3\times3}\longrightarrow[0,+\infty]$
and the rescaled applied loads
$f_h:\Om\longrightarrow\R^3$ as
\begin{equation}
\label{eq:resc_ener}
W^h(x_3,F)
\,:=\,
w^h(hx_3,F),
\qquad\qquad\quad
f_h(x)
\,:=\,
\hat f_h(x',hx_3).
\end{equation}
Note that $W^h$ fulfills
\begin{equation*}
W^h(x_3,F)=0\qquad\quad\mbox{iff}
\qquad\quad
F\in{\rm SO}(3)\sqrt{\overline C_h(z_3)},
\qquad\qquad
\overline C_h(x_3):=\bar c_h(hx_3).
\end{equation*}
Setting
\begin{equation}
\label{resc_grad}
\na_hy:=
\left(
\begin{array}{ccccc}
\!\!\partial_{x_1}y\!\! & \!\!\bigg|\!\! 
& \!\!\partial_{x_2}y\!\! & \!\!\bigg|\!\! 
& \!\!\displaystyle\frac{\partial_{x_3}y}h\!\!
\end{array}
\right)
=:
\left(
\begin{array}{ccccc}
\!\!\na'y\!\!
& \!\!\bigg|\!\! 
& \!\!\displaystyle\frac{\partial_{x_3}y}h\!\!
\end{array}
\right),
\qquad\qquad
\mbox{for every}\quad y:\Om\longrightarrow\R^3,
\end{equation}
the correspondence between the original quantities and the
rescaled ones is through the formulas
\begin{equation}
\label{correspondence}
\hat{\mathscr E}^h(v)
\,=\,
h\,\mathscr E^h(y),
\qquad\quad
\hat{\mathscr F}^h(v)
\,=\,
h\,\mathscr F^h(y),
\qquad\qquad
v(z):=y(z',z_3/h)
\quad\mbox{a.e.}\ z\ \mbox{in}\ \Omega_h.
\end{equation}
Here, the rescaled stored elastic energy functional $\mathscr E^h$ 
and the rescaled total energy functional $\mathscr F^h$
are defined, on a deformation 
$y:\Om\to\R^3$, 
as
\begin{equation}
\label{mathscr_E_h}
\mathscr E^h(y)
:=
\int\limits_{\Om}W^h(x_3,\na_h y(x))\,\dd x,
\qquad\qquad
\mathscr F^h(y)
\,:=\,
\mathscr E^h(y)
-
\int\limits_{\Om}f_h\cdot y\,\dd x.
\end{equation}
Following the notation already introduced in 
Section \ref{subsec:three_dim_model},
we use the indexes $SB$ and $T$ to denote the
quantities related to the splay-bend case and twisted case,
respectively. 
Hence, we write $\overline C_{h,SB}$, $W^h_{SB}$,
$\mathscr E^h_{SB}$, and $\mathscr F^h_{SB}$   
for the splay-bend model, 
and $\overline C_{h,T}$, $W^h_T$,
$\mathscr E^h_T$, and $\mathscr F^h_T$   
for the twisted model.

We now focus attention on the (rescaled)
spontaneous strains $\overline C_h(x_3)$.
Looking at \eqref{director_splay}, we first note that
for both models $n^h(hx_3)$ is independent of $h$, namely
\begin{equation}
\label{tensor_N}
N_{SB}(x_3)
:=
n^h_{SB}(hx_3)
=
\left(
\begin{array}{c}
\cos\big(\frac{\pi}4+\frac\pi2x_3\big)\\
0\\
\sin\big(\frac{\pi}4+\frac\pi2x_3\big)
\end{array}
\right),
\qquad\quad
N_T(x_3)
:=
n^h_T(hx_3)
=
\left(
\begin{array}{c}
\cos\big(\frac{\pi}4+\frac\pi2x_3\big)\\
\sin\big(\frac{\pi}4+\frac\pi2x_3\big)\\
0
\end{array}
\right),
\end{equation}
for every $x_3\in(-1/2,1/2)$.
Hence, referring to the (above) definition of $\overline C_h$
and to expression \eqref{bar_c_h}, 
we have for the splay-bend case as well as for the twisted case 
\begin{align}
\label{eq:resc_SS}
\overline C_h(x_3)
&
=a_h^{2/3}N(x_3)\otimes N(x_3)
+a_h^{-1/3}\big(
{\rm I}
-N(x_3)\otimes N(x_3)\big)\nonumber\\
&
=\big(a_h^{2/3}-a_h^{-1/3}\big)
\left(\frac{{\rm I}}{a_h-1}+N(x_3)\otimes N(x_3)\right)\nonumber\\
&
={\rm I}
+\frac{\alpha_0h}{h_0}\left(N(x_3)\otimes N(x_3)-
\frac{{\rm I}}3\right)
+R^h(x_3),
\end{align}
where 
$\|R^h\|_{\infty}=o(h)$
and $\|\cdot\|_{\infty}$ is the norm in the space
${\rm L}^{\infty}\big((-1/2,1/2),\R^{3\times3}\big)$.
Note that in the third equality we have
plugged in expression \eqref{a_h} for $a_h$
and used the expansion
\begin{equation*}
a_h^{2/3}-a_h^{-1/3}
=
\frac{\alpha_0h}{h_0}
-\frac13\left(\frac{\alpha_0h}{h_0}\right)^2
+o(h^2).
\end{equation*}


\subsection{A rigorous mathematical result for the limiting theory}
\label{rigorous_res}


For the convenience of the reader, we collect in this section,
in a slightly simplified version,
two results proved in \cite{Schmidt2007}
(Theorems \ref{compattezza} and \ref{thm_Schmidt_Gamma} below), 
which we are going to use later on.
In this paper, 
an arbitrary family of energy densities
$W^h:(-1/2,1/2)\times\R^{3\times3}\longrightarrow[0,+\infty]$
is considered, with the property that 
\begin{equation}
\label{cond_W_h_cruc}
W^h(x_3,F)
\,=\,
W_0\big(F(
{\rm I}
+hB^h(x_3))\big),
\end{equation}
where the function $W_0:\R^{3\times3}\longrightarrow[0,+\infty]$ 
satisfies Assumption \ref{assunzione_W_0}
below, and 
\begin{equation*}
B^h\longrightarrow B
\qquad\qquad 
\mbox{in}
\qquad 
{\rm L}^{\infty}((-1/2,1/2),\R^{3\times3}),
\qquad
\mbox{as}\quad h\downarrow 0.
\end{equation*}
For each small $h$, let us introduce the functional 
$\mathscr E^h:{\rm W}^{1,2}(\Omega,\R^3)\to[0,+\infty]$,
defined as 
\begin{equation*}
\mathscr E^h(y)
\;:=\,
\int\limits_{\Omega}W^h(x_3,\na_hy){\rm d}x,
\end{equation*}
with $\na_h$ given by \eqref{resc_grad}.
Recall that here and throughout the paper
$\Omega=\omega\times(-1/2,1/2)$ and $\omega\subset\R^2$
is a bounded Lipschitz domain with sufficiently regular boundary.
More precisely, for the following theorems
to hold, it is required that 
there exists a closed subset $\Sigma\subset\pa\omega$
with $\mathscr H^1(\Sigma)=0$ such that the outer 
unit normal exists and 
is continuous on $\pa\omega\setminus\Sigma$.

\begin{assumption}
\label{assunzione_W_0}
The function $W_0:\R^{3\times3}\longrightarrow[0,+\infty]$ 
fulfills the following conditions: 

\begin{itemize}

\smallskip
\item[(i)]
\ it is ${\rm C}^2$ in a neighborhood of $\SO(3)$,
and it is minimised at ${\rm I}$;

\smallskip
\item[(ii)]
\ it is frame-indifferent, i.e. 
$W_0(F)=W_0(RF)$ for every $R\in\SO(3)$.

\smallskip
\item[(iii)]
\ there exists a constant $C>0$ such that for every $F\in\RRR$, 
\begin{equation*}
W_0(F)
\,\geq\, 
C\,\dist^2\big(F,\SO(3)\big).
\end{equation*}

\end{itemize}

\end{assumption}

The following result states that a sequence $\{y_h\}$ which
bounds the energy $\mathscr E^h$ by a factor $h^2$
converges (up to subsequences) to a limit
that is constrained to the class of (${\rm W}^{2,2}$-)
isometric immersions of $\omega$ into the three-dimensional
Euclidean space, namely
\begin{equation}
\label{Aiso}
\mathcal A_{iso}
\,:=\,
\Big\{
y\in{\rm W}^{2,2}(\omega,\R^3):(\na'y)^T\na'y=
{\rm I}_2
\ \mbox\ {a.e.\ in}\quad\omega
\Big\}.
\end{equation}

\begin{theorem}[Compactness]
\label{compattezza}
If $\{y^h\}\subset{\rm W}^{1,2}(\Omega,\R^3)$
is a sequence such that
\begin{equation}
\label{hyp_comp}
\int\limits_{\Omega}W^h(x_3,\na_hy_h){\rm d}x\leq Ch^2
\end{equation}
for every $h>0$ small, then there exists a 
(not relabelled) subsequence such that
\begin{equation*}
\na_hy^h\longrightarrow
\left(
\begin{array}{ccccc}
\!\!\na'y\!\! 
& \!\!\big|\!\! 
& \!\!\nu\!\! 
\end{array}
\right),
\qquad\mbox{as}\quad 
h\downarrow0,
\quad
\text{strongly in}\quad{\rm L}^2(\Om,\RRR).
\end{equation*}
Moreover,
the function
$x\mapsto\left(
\begin{array}{ccc}
\!\!\na'y\!\! 
& \!\!\big|\!\! 
& \!\!\nu\!\! 
\end{array}
\right)$
belongs to ${\rm W}^{1,2}(\Omega,\R^{3\times3})$,
is independent of $x_3$, and\\
$\left(
\begin{array}{ccc}
\!\!\na'y\!\! 
& \!\!\big|\!\! 
& \!\!\nu\!\! 
\end{array}
\right)(x')\in\SO(3)$ for a.e. $x'\in\omega$.
\end{theorem}

Before proceeding, let us introduce some more notation 
and denote by $Q_3(M)$, $M\in\R^{3\times3}$, 
the quadratic form 
$\rmD^2W_0({\rm I})[M,M]$, where
$\rmD^2W_0({\rm I})$ stands for the second differential
of $W_0$ evaluated at ${\rm I}$.
Moreover, define, for every $G\in\R^{2\times 2}$, 
\begin{equation}
\label{Q2}
Q_2(G)
\,:=\,
\min_{\stackrel{b\in\R^2}{a\in\R}}
Q_3
\left(
\left[
\begin{tabular}{c|c}
$G$ & $b$\\
\hline
$0$ & $a$
\end{tabular}
\right]
\right),
\end{equation}
and in turn
\begin{equation}
\label{barQ2}
\overline Q_2(G)
\,:=\,
\min_{D\in\R^{2\times2}}\int\limits_{-1/2}^{1/2}
Q_2\big(D+t\,G+\check B(t)\big){\rm d}t,
\end{equation}
where $\check B$ is obtained from $B$ by omitting the last row
and the last column.

\begin{theorem}[$\Gamma$-convergence]
\label{thm_Schmidt_Gamma}
The functionals 
$\mathscr E^h/h^2$ $\Gamma$-converge as $h\downarrow0$,
with respect to the strong and the weak topology of 
${\rm W}^{1,2}(\Omega,\R^3)$,
to
\begin{equation*}
\mathscr E^{lim}(y)
\,:=\,
\left\{
\begin{array}{ll}
\displaystyle\frac12\int\limits_{\omega}\overline Q_2(A_y(x')){\rm d}x'
&\qquad\mbox{if}\quad y\in\mathcal A_{iso},\\
\,& \,\\
+\infty &\qquad\mbox{otherwise in}\quad{\rm W}^{1,2}(\Omega,\R^3),
\end{array}\right.
\end{equation*}
where $A_y$ denotes the second fundamental
form associated with the surface $y(\omega)$.
\end{theorem}
\noindent Recall that the second fundamental form of $y(\omega)$
at a point $y(x')$ can be expressed as
$\big(\na'y(x')\big)^T\na'\nu(x')$,
where $\nu:=\pa_{x_1}y\wedge\pa_{x_2}y$.


\subsection
{Splay-bend and twisted nematic elastomer plates}
\label{nostra_derivazione}


We want to apply the theory presented in the previous section
to our two models.
We first focus on the splay-bend case,
whose associated rescaled stored energy density,
considering expression \eqref{prot_energy} together with\eqref{eq:resc_ener} and \eqref{eq:resc_SS},
is given, for every $x_3\in(-1/2,1/2)$ and 
every $F\in\R^{3\times3}$ with $\det F>0$, by
\begin{equation*}
W^h_{SB}(x_3,F)
\,=\,
\frac{\mu}2\Big[(F^TF)\cdot
\overline C_{h,SB}^{-1}(x_3)
-3-2\log(\det F)\Big]
+W_{vol}(\det F).
\end{equation*}
Recall that
\begin{equation}
\label{SS_SB}
\overline C_{h,SB}(x_3)
=
{\rm I}
+\frac{\alpha_0h}{h_0}
\left(N_{SB}(x_3)\otimes N_{SB}(x_3)-
\frac{{\rm I}}3\right)
+R^h_{SB}(x_3),
\qquad\quad
\|R^h_{SB}\|_{\infty}=o(h)
\end{equation}
(see \eqref{tensor_N} for the definition of $N_{SB}$).
Defining
\begin{equation}
\label{nostra_W}
W_0(F)
\,:=\,
\frac{\mu}2
\Big[
|F|^2-3-2\log(\det F\,)
\Big]
+
W_{vol}(\det F\,),
\end{equation}
for every $F\in\R^{3\times3}$ with $\det F>0$,
and setting
\begin{equation}
\label{delta_0}
\overline U_{h,SB}
:=\sqrt{\overline C_{h,SB}},
\end{equation}
yields
$W^h_{SB}(x_3,F)=W_0\big(F\overline U_{h,SB}^{-1}(x_3)\big)$.
Note that  
\begin{equation*}
\overline U_{h,SB}^{-1}
=
{\rm I}
+h\left[-\delta_0\left(M_{SB}-
\frac{{\rm I}}3\right)+
  \frac{r^h_{SB}}h\right] ,    
\end{equation*}
with $\|r^h_{SB}\|_{\infty}=o(h)$,
where we have used the notation
\begin{equation}
\label{eq:cose}
\delta_0:=\frac{\alpha_0}{2h_0},
\qquad\qquad
M_{SB}
:=
N_{SB}\otimes N_{SB}
=
\left(
\begin{array}{ccc}
\cos^2f_1 & 0 & 
\frac12\sin(2f_1)\\
0 & 0 & 0\\
\frac12\sin(2f_1) 
    & 0 & 
\sin^2f_1
\end{array}
\right).
\end{equation}
Here, the function $x_3\mapsto f_1(x_3)$
is defined as in \eqref{f_h}, with $h=1$.
All in all, we can write
\begin{equation}
\label{W_h_SB}
W^h_{SB}(x_3,F)
=
W_0\big(F(
{\rm I}
+hB^h_{SB}(x_3))\big),
\qquad\qquad
B^h_{SB}:=\left[-\delta_0\left(M_{SB}-
\frac{{\rm I}}3\right)+
  \frac{r^h_{SB}}h\right].
\end{equation}
Since $\|r^h_{SB}\|_{\infty}=o(h)$, we have that
$B^h_{SB}\longrightarrow B_{SB}$
in
${\rm L}^{\infty}\big((-1/2,1/2),\R^{3\times 3}\big)$,
where
$
B_{SB}:=-\delta_0\left(M_{SB}(x_3)-
\frac{{\rm I}}3\right)
$.
In turn, also in view of Remark \ref{rmk_W_0} below,
we have shown that the splay-bend model introduced 
in Section \ref{subsec:three_dim_model}
perfectly fits the mathematical theory 
summarized in the previous section.
Hence, 
we have to compute the 2D energy density according to 
formula \eqref{barQ2}.
First of all, we have that
$
Q_3(M)
\,=\,
2\mu\,|\sym M|^2+W_{vol}''(1)\,\tr^2M.
$
Using this expression, 
we can compute $Q_2$ for every $G\in\R^{2\times2}$ (see \eqref{Q2}): 
\begin{align}
\label{Q2nostra}
Q_2(G)
&\,=\,
\min_{\stackrel{b\in\R^2}{a\in\R}}
\left\{
2\mu\left|\left(
\begin{tabular}{c|c}
$\sym G$ & $b/2$\nonumber\\
\hline
$b^T/2$ & $a$
\end{tabular}
\right)\right|^2
+
W_{vol}''(1)
(\tr\,G+a)^2
\right\}\\
&\,=\,
2\mu\,|\sym G|^2+W_{vol}''(1)\,\tr^2G+
\min_{a\in\R}
\Big[\big(2\mu+W_{vol}''(1)\big)a^2+
2\,W_{vol}''(1)\tr\,G\,a\Big]
\,=\,
2\mu\left(|\sym G|^2+\gamma\,\tr^2G\right),
\end{align}
having introduced the notation
\begin{equation}
\label{def:lambda}
\gamma:=
\frac{W_{vol}''(1)}{2\mu+W_{vol}''(1)}.
\end{equation}
Finally (cfr \eqref{barQ2}), 
note that $\check B_{SB}$ is given by 
\begin{equation*}
\check B_{SB}=-\delta_0\left(\check M_{SB}-
\frac{{\rm I}_2}3\right),
\qquad
\mbox{with}
\qquad
\check M_{SB}(x_3):=
\big(N_{SB}(x_3)\otimes N_{SB}(x_3)\big)^{\check{}}=
\left(
\begin{array}{cc}
\cos^2\big(\frac{\pi}4+\frac\pi2x_3\big) & 0\\
0 & 0
\end{array}
\right).
\end{equation*}
We are now in the position to compute,
for every $G\in\R^{2\times 2}$,
\begin{align*}
\overline Q_{2,SB}(G)
&:=
\min_{D\in\R^{2\times2}}\int\limits_{-1/2}^{1/2}
Q_2\left(D+t\,G+\check B_{SB}(t)\right){\rm d}t\\
&=
2\mu
\min_{D\in\Sym(2)}\int\limits_{-1/2}^{1/2}
\!\!\left\{
\left|D+t\,\sym G-\delta_0\,\check M_{SB}(t)+\frac{\delta_0}3
{\rm I}_2
\right|^2
\!\!+
\gamma\,
\tr^2\!
\left(D+t\,G-\delta_0\,\check M_{SB}(t)+\frac{\delta_0}3
{\rm I}_2\right)
\!\!\right\}{\rm d}t.
\end{align*}
The integrals
\begin{equation*}
\int\limits_{-1/2}^{1/2}|\check M_{SB}|^2{\rm d}t
=
\int\limits_{-1/2}^{1/2}
\cos^4\left(\frac{\pi}4+\frac\pi2t\right){\rm d}t
=\frac 38,
\qquad
\int\limits_{-1/2}^{1/2}t\,\tr\,\check M_{SB}{\rm d}t
=
\int\limits_{-1/2}^{1/2}t
\cos^2\left(\frac{\pi}4+\frac\pi2t\right){\rm d}t
=
-\frac1{\pi^2}
\end{equation*}
and other elementary computations imply that
$\overline Q_{2,SB}(G)/(2\mu)$ equals
\begin{equation*}
\frac1{12}\Big(|\sym G|^2+\gamma\,\tr^2G\Big)
+
\frac{2\,\delta_0}{\pi^2}
\Big(\sym G\cdot{\rm diag}(1,0)+\gamma\,\tr\,G\Big)
+
\left(\frac{19+11\,\gamma}{72}\right)\delta_0^2
+
\min_{D\in\Sym(2)}q_{SB}(D),
\end{equation*}
where
\begin{equation*}
q_{SB}(D):=|D|^2+\gamma\,\tr^2D-
\delta_0\left[\sym D\cdot{\rm diag}(1,0)
+\Big(\frac{2+\gamma}3\Big)\tr\, D\right].
\end{equation*}
It is easy to see that
\begin{equation*}
\min_{D\in\Sym(2)}q_{SB}(D)
=
q_{SB}\big({\rm diag}(\delta_0/6,-\delta_0/3)\big)
=
-\Big(\frac{5+\gamma}{36}\Big)\delta_0^2,
\end{equation*}
and in turn that
\begin{equation*}
\overline Q_{2,SB}(G)=2\mu\bigg[
\frac1{12}\Big(|\sym G|^2+\gamma\,\tr^2G\Big)
+
\frac{2\delta_0}{\pi^2}
\Big(\sym G\cdot{\rm diag}(1,0)+\gamma\,\tr\, G\Big)
+
\Big(\frac{1+\gamma}8\Big)\delta_0^2
\bigg].
\end{equation*}
It is again a simple computation showing that there exist constants 
$\alpha_{SB},\,\beta_{SB}\in\R$ and $\overline A_{SB}\in\Sym(2)$
such that
\begin{equation}
\label{barQ_2_SB}
\overline Q_{2,SB}(G)
\,=\,
\alpha_{SB}\,Q_2[G-\overline A_{SB}]^2+\beta_{SB},
\qquad\quad
\mbox{for every}\quad G\in\Sym(2),
\end{equation}
and they are given by
\begin{equation}
\label{const_SB}
\alpha_{SB}\,=\,\frac1{12},
\quad\qquad
\overline A_{SB}\,=\,\frac{12\delta_0}{\pi^2}\,{\rm diag(-1,0)},
\quad\qquad
\beta_{SB}\,=\,\mu\,(1+\gamma)\delta_0^2\Big(\frac{\pi^4-12}4\Big).
\end{equation}

To state our result, let us define the functional
$\mathscr E^{lim}_{SB}:
\mathcal A_{iso}\longrightarrow[0,\infty)$,
where $\mathcal A_{iso}$ is the class defined in \eqref{Aiso},
as
\begin{align}
\mathscr E^{lim}_{SB}(y)
&:=
\frac12
\int\limits_{\omega}\overline Q_{2,SB}(A_y(x')){\rm d}x'\nonumber\\
&=
\frac{\mu}{12}\int\limits_{\omega}
\!\!\left\{\Big|{\rm A}_y(x')\!-\!\frac{12\delta_0}{\pi^2}\,
{\rm diag}\big(-1,0\big)\Big|^2
\!\!+\gamma
\Big({\rm H}_y(x')\!+\!
\frac{12\delta_0}{\pi^2}\Big)^2\!\right\}\!\dd x'
+
\mu\,(1+\gamma)\delta_0^2
\Big(\frac{\pi^4-12}8\Big)\!|\omega|.
\label{2D_energy_functional_SB}
\end{align}
Here, the symbol ${\rm H}_y$ 
denotes the mean curvature of $y(\omega)$,
hence ${\rm H}_y=\tr A_y$.
Note that for every $y\in\mathcal A_{iso}$
we have that $|{\rm A}_y|\in{\rm L}^2(\omega)$,
and in turn $\mathscr E^{lim}_{SB}(y)<+\infty$.

Theorems 
\ref{compattezza} and \ref{thm_Schmidt_Gamma}
and standard results of the theory of $\Gamma$-convergence,
tell us that $3D$ low-energy sequences converge,
up to subsequences, to a minimiser of the
derived 2D model.
This is the content of the following theorem. 
We refer the reader to \eqref{mathscr_E_h} and the subsequent paragraph
for the definition of the 3D total-energy functionals $\mathscr F_{SB}^h$.

\begin{theorem}
[Splay-bend plate model]
\label{main_thm_SB}
Suppose that the rescaled loads $f_h$ 
are such that $f_h/h^2\rightharpoonup f$
weakly in ${\rm L}^2(\Om,\R^3)$ and satisfy the
normalizing condition $\int_{\Om}f_h\,\dd x=0$.
Define
the 2D total energy functional
$\mathscr F_{SB}^{lim}:\mathcal A_{iso}\longrightarrow\R$
as 
\begin{equation*}
\mathscr F_{SB}^{lim}(y)
\,:=\,
\mathscr E_{SB}^{lim}(y)
-\int\limits_{\omega} f^{lim}(x')\cdot y(x')\,\dd x',
\end{equation*}
where 
$\mathscr E_{SB}^{lim}$ is defined as in \eqref{2D_energy_functional_SB}
and 
$f^{lim}(x')
\,:=\,
\int_{-1/2}^{1/2}f(x',x_3)\,\dd x_3,$
for a.e. $x'\in\omega.$
Suppose that $\{y_h\}$ is a low-energy sequence, viz.
\begin{equation*}
\lim_{h\to0}\frac{\mathscr F_{SB}^h(y_h)}{h^2}
\,=\,
\lim_{h\to0}\frac
{\inf_{{\rm W}^{1,2}(\Omega,\R^3)}\mathscr F^h_{SB}}
{h^2}
\,=:\,
m.
\end{equation*}
Then, up to a subsequence, 
$y_h\longrightarrow y_{SB}$ in 
${\rm W}^{1,2}(\Om,\R^3)$, where
$y_{SB}\in\mathcal A_{iso}$ is a minimiser of the
2D model, that is 
\[\mathscr F_{SB}^{lim}(y_{SB})
\,=\,
\min_{\mathcal A_{iso}}\mathscr F_{SB}^{lim}.
\]
Moreover, $m=\mathscr F_{SB}^{lim}(y_{SB}).$
\end{theorem}

If we let $f=0$ in the above theorem, we have
\begin{equation*}
\min_{\mathcal A_{iso}}\mathscr F_{SB}^{lim}
\,=\,
\min_{\mathcal A_{iso}}\mathscr E_{SB}^{lim}
\,=\,
\mathscr E_{SB}^{lim}(y_{SB})
\,=\,\mu\,(1+\gamma)\delta_0^2\Big(\frac{\pi^4-12}8\Big)\!|\omega|,
\end{equation*}
and the associated fundamental form of $y_{SB}$
is given by $(12\delta_0/\pi^2)\,{\rm diag}(-1,0)$.

Let us now fix a low-energy sequence $\{y_h\}$ 
converging to a minimiser $y\in\mathcal A_{iso}$
and rephrase the theorem in terms of the physical
total energies $\hat{\mathscr F}_{SB}^h$ 
defined in \eqref{phys_quan}. 
Defining the deformations $v_h(z',z_3)=y_h(z',z_3/h)$
in the physical reference configuration $\Omega_h$,
we have
$\lim_{h\to0}\hat{\mathscr F}_{SB}^h(v_h)/h^3
=
\min_{\mathcal A_{iso}}\mathscr F_{SB}^{lim},
$
in view of \eqref{correspondence}.
Equivalently, for a given small thickness $h_0$,
the approximate identity
\begin{multline}
\label{phys_stat_splay}
\hat{\mathscr F}_{SB}^{h_0}(v_{h_0})
\,\cong\,
\frac{\mu\,h_0^3}{12}\int\limits_{\omega}
\!\!\left\{\Big|{\rm A}_y(x')\!+\!\frac{12\delta_0}{\pi^2}\,
{\rm diag}\big(1,0\big)\Big|^2
\!\!+\gamma
\Big({\rm H}_y(x')\!+\!
\frac{12\delta_0}{\pi^2}\Big)^2\!\right\}\!\dd x'\\
+
\mu\,h_0^3\,(1+\gamma)\delta_0^2
\Big(\frac{\pi^4-12}8\Big)\!|\omega|
-h_0^3\int\limits_{\omega} f^{lim}(x')\cdot y(x')\,\dd x'
\end{multline}
holds true,
modulo terms of order higher than $3$ in $h_0$.

\begin{remark}
\label{rmk_W_0}
\upshape
Clearly, the function $W_0$ defined in \eqref{nostra_W}
vanishes in $\SO(3)$. 
Also, by the standard inequality between arithmetic and geometric mean
we have that $|F|^2\geq3(\det F)^{2/3}$ for every $F\in\R^{3\times3}$
with positive determinant, which proves that
\begin{equation*}
W_0(F)\,\geq\,\frac{3\mu}2\,\psi\bigg(\frac{|F|^2}3\bigg),
\qquad\qquad
\psi(t)\,:=\,t-1-\log t, \quad t>0.
\end{equation*}
In particular, we have that $W_0(F)=0$ iff $F\in\SO(3)$
and that $W_0(F)\geq C|F|^2$ for every large $|F|$.
Moreover, due to the regularity of $W_0$ around $\SO(3)$, 
the energy density grows quadratically close to $\SO(3)$.
These facts show that $W_0$ satisfies Assumption \ref{assunzione_W_0}.
\end{remark}


\bigskip

We now move to the twisted geometry.
In this case, the (renormalized) spontaneous
strain distribution is given by 
\begin{equation}
\label{SS_T}
\overline C_{h,T}(x_3)
=
{\rm I}
+\frac{\alpha_0h}{h_0}
\left(N_T(x_3)\otimes N_T(x_3)-
\frac{{\rm I}}3\right)
+R^h_T(x_3),
\qquad\quad
\|R^h_T\|_{\infty}=o(h),
\end{equation}
where $N_T$ is defined as in \eqref{tensor_N}, and 
the (rescaled) stored energy density is
\begin{equation*}
W^h_T(x_3,F)
\,=\,
\frac{\mu}2\Big[(F^TF)\cdot
\overline C_{h,T}^{-1}(x_3)
-3-2\log(\det F)\Big]
+W_{vol}(\det F)
\end{equation*}
on every deformation gradient 
$F\in\R^{3\times3}$ such that $\det F>0$.
Proceeding similarly to the splay-bend case,
we set
$
\overline U_{h,T}
:=\sqrt{\overline C_{h,T}}
$,
so that
$W^h_T(x_3,F)=W_0\big(F\overline U_{h,T}^{-1}(x_3)\big)$,
being $W_0$ defined as in \eqref{nostra_W}.
Note that, by Taylor-expanding $\sqrt{\overline C_{h,T}}$
around ${\rm I}$, we get  
\begin{equation*}
\overline U_{h,T}^{-1}
=
{\rm I}
+h\left[-\delta_0\left(M_T-
\frac{{\rm I}}3\right)+
  \frac{r^h_T}h\right],
  \qquad
M_T
:=
N_T\otimes N_T
=
\left(
\begin{array}{ccc}
\cos^2f_1 & \frac12\sin(2f_1) & 0 \\
\frac12\sin(2f_1) &  \sin^2f_1 & 0 \\
0 & 0 & 0\\
\end{array}
\right).    
\end{equation*}
where $\|r^h_T\|_{\infty}=o(h)$,
the positive constant $\delta_0$ is defined as in \eqref{eq:cose},
and where $x_3\mapsto f_1(x_3)$
given by \eqref{f_h} (with $h=1$).
Hence, we can write
\begin{equation}
\label{W_h_T}
W^h_T(x_3,F)
=
W_0\big(F(
{\rm I}
+hB^h_T(x_3))\big),
\qquad\qquad
B^h_T:=\left[-\delta_0
\left(M_T-\frac{{\rm I}}3\right)+
  \frac{r^h_T}h\right],
\end{equation}
and we have that
$B^h_T\longrightarrow B_T$
in
${\rm L}^{\infty}\big((-1/2,1/2),\R^{3\times 3}\big)$,
where
$
B_T:=-\delta_0\left(M_T(x_3)-\frac{{\rm I}}3\right)
$.
Now, arguing as for the splay-bend case
and using \eqref{Q2nostra}--\eqref{def:lambda},
we are left to derive (cfr \eqref{barQ2}) 
the expression for
$\check B_T=-\delta_0\left(\check M_T-
\frac{{\rm I}_2}3\right)$,
where
\begin{equation*}
\check M_T(x_3)
:=
\big(N_T(x_3)\otimes N_T(x_3)\big)^{\check{}}
=
\left(
\begin{array}{cc}
\cos^2\big(\pi/4+\pi x_3/2\big) & 
\frac12\sin\big(\pi/2+\pi x_3\big)\\
\frac12\sin\big(\pi/2+\pi x_3\big) & 
\sin^2\big(\pi/4+\pi x_3/2\big)
\end{array}
\right),
\end{equation*}
and to compute, for every $G\in\R^{2\times 2}$,
\begin{align*}
\overline Q_2(G)
&:=
\min_{D\in\R^{2\times2}}\int\limits_{-1/2}^{1/2}
Q_2\big(D+t\,G+\check B_T(t)\big){\rm d}t\\
&=
2\mu
\min_{D\in\Sym(2)}\int\limits_{-1/2}^{1/2}
\!\!\left\{
\left|D+t\,\sym G-\delta_0\,\check M_T(t)+\frac{\delta_0}3
{\rm I}_2
\right|^2
\!\!+
\gamma\,
\tr^2\!
\left(D+t\,G-\delta_0\,\check M_T(t)+\frac{\delta_0}3
{\rm I}_2\right)
\!\!\right\}{\rm d}t.
\end{align*}

The integrals

\begin{equation*}
\int\limits_{-1/2}^{1/2}
\!\!\cos^2\left(\frac{\pi}4+\frac{\pi}2t\right)
{\rm d}t
\,\,=
\int\limits_{-1/2}^{1/2}
\!\!\sin^2\left(\frac{\pi}4+\frac{\pi}2t\right)
=
\frac12,
\qquad
\int\limits_{-1/2}^{1/2}
\!\!\sin\left(\frac{\pi}2+\pi t\right){\rm d}t
\,=\,\frac2{\pi},
\end{equation*}
and
\begin{equation*}
\int\limits_{-1/2}^{1/2}
\!\!t\cos^2\left(\frac{\pi}4+\frac{\pi}2t\right)
{\rm d}t
\,=\,
-\!\!\int\limits_{-1/2}^{1/2}
\!\!t\sin^2\left(\frac{\pi}4+\frac{\pi}2t\right)
\,=\,
-\frac1{\pi^2},
\qquad
\int\limits_{-1/2}^{1/2}
\!\!t\,\sin\left(\frac{\pi}2+\pi t\right){\rm d}t
\,=\,0,
\end{equation*}
give
\begin{equation*}
\int\limits_{-1/2}^{1/2}\!\!\check M_T\,{\rm d}t
\,\,=\,
\left(
\begin{array}{cc}
1/2 & 1/\pi\\
1/\pi & 1/2
\end{array}
\right)
\qquad
\mbox{and}
\qquad
\int\limits_{-1/2}^{1/2}
\!\!t\,\check M_T\,{\rm d}t
\,\,=\,
\left(
\begin{array}{cc}
-1/\pi^2 & 0\\
0 & 1/\pi^2
\end{array}
\right).
\end{equation*}
These computations, together with the fact that
$\tr\,\check M_T=|\check M_T|=1$, 
show that
$\overline Q_{2,T}(G)/(2\mu)$ equals
\begin{equation*}
\frac1{12}\Big(|\sym G|^2+\gamma\,\tr^2G\Big)
+
\frac{2\,\delta_0}{\pi^2}\,
\sym G\cdot{\rm diag}(1,-1)
+
\left(\frac{5+\gamma}9\right)\delta_0^2
+
\min_{D\in\Sym(2)}q_T(D),
\end{equation*}
where
\begin{equation*}
q_T(D):=|D|^2+\gamma\,\tr^2D-
\delta_0
\left[\sym D\cdot
\left(
\begin{array}{cc}
1 & 2/\pi \\
2/\pi & 1
\end{array}
\right)
-\frac23(1-\gamma)\tr\, D\right].
\end{equation*}
It is easy to see that
\begin{equation*}
\min_{D\in\Sym(2)}q_T(D)
=
q_T\left(
\left[
\begin{array}{cc}
\delta_0/6 & \delta_0/\pi\\
\delta_0/\pi & \delta_0/6
\end{array}
\right]
\right)
=
-\Big(\frac{1+2\gamma}{18}+\frac2{\pi^2}\Big)\delta_0^2,
\end{equation*}
and in turn that
\begin{equation*}
\overline Q_{2,T}(G)=2\mu\bigg[
\frac1{12}\Big(|\sym G|^2+\gamma\,\tr^2G\Big)
+
\frac{2\,\delta_0}{\pi^2}\,
\sym G\cdot{\rm diag}(1,-1)
+
\Big(\frac12-\frac2{\pi}\Big)\delta_0^2
\bigg].
\end{equation*}
Other straightforward computations show that, setting
\begin{equation}
\label{const_T}
\alpha_T\,=\,\frac1{12},
\quad\qquad
\overline A_T\,=\,\frac{12\delta_0}{\pi^2}\,{\rm diag(-1,1)},
\quad\qquad
\beta_T\,=\,\mu\,\Big(\frac{\pi^4-4\pi^2-48}{\pi^4}\Big)\delta_0^2,
\end{equation}
one has
\begin{equation}
\label{barQ_2_T}
\overline Q_{2,T}(G)
\,=\,
\alpha_T\,Q_2[G-\overline A_T]^2+\beta_T,
\qquad\quad
\mbox{for every}\quad G\in\Sym(2).
\end{equation}
To state the result pertaining to the twisted model,
we define the functional
$\mathscr E^{lim}_T:
\mathcal A_{iso}\longrightarrow[0,\infty)$,
where the class 
$\mathcal A_{iso}$ is defined in \eqref{Aiso},
as
\begin{align}
\mathscr E^{lim}_T(y)
&:=
\frac12
\int\limits_{\omega}\overline Q_{2,T}(A_y(x')){\rm d}x'\nonumber\\
&=
\frac{\mu}{12}\int\limits_{\omega}
\left\{\Big|{\rm A}_y(x')-\frac{12\delta_0}{\pi^2}
{\rm diag}\big(-1,1\big)\Big|^2
+\gamma\,
{\rm H}_y^2(x')\right\}\dd x'
+
\frac{\mu\,\delta_0^2}{\pi^4}
\Big(\frac{\pi^4-4\pi^2-48}2\Big)|\omega|.
\label{2D_twist_model}
\end{align}
We recall that $\mu$ and $\gamma$ are the elastic constants
appearing in \eqref{prot_energy} and 
defined in \eqref{def:lambda}, respectively.
As for the splay-bend case, well-known results
of the theory of $\Gamma$-convergence easily
imply the following theorem.
We refer to \eqref{mathscr_E_h} and to the subsequent paragraph
for the notation related to the 3D models.

\begin{theorem}
[Twisted plate model]
\label{main_thm_T}
Under the same assumptions on the family of (rescaled)
loads $\{f_h\}$ as in Theorem \ref{main_thm_SB},
define 
$\mathscr F_T^{lim}:\mathcal A_{iso}\longrightarrow\R$ as
\begin{equation}
\label{2D_total_energy_T}
\mathscr F_T^{lim}(y)
\,:=\,
\mathscr E_T^{lim}(y)
-\int\limits_{\omega} f^{lim}(x')\cdot y(x')\,\dd x',
\end{equation}
where 
$\mathscr E_T^{lim}$ is given by \eqref{2D_twist_model}
and 
$f^{lim}(x')
:=
\int_{-1/2}^{1/2}f(x',x_3)\,\dd x_3,$
for a.e. $x'\in\omega.$
Suppose that $\{y_h\}$ is a low-energy sequence, viz.
\begin{equation*}
\lim_{h\to0}\frac{\mathscr F_T^h(y_h)}{h^2}
\,=\,
\lim_{h\to0}\frac{\inf_{y\in{\rm W}^{1,2}(\Omega,\R^3)}\mathscr F^h_T}{h^2}
\,=:\,
m.
\end{equation*}
Then, up to a subsequence, $y_h\longrightarrow y_T$ in 
${\rm W}^{1,2}(\Om,\R^3)$, where $y_T\in\mathcal A_{iso}$
is a minimiser of the 2D model, that is 
\begin{equation*}
\mathscr F_T^{lim}(y_T)\,=\,\min_{\mathcal A_{iso}}\mathscr F_T^{lim}.
\end{equation*}
Moreover, $m=\mathscr F_T^{lim}(y_T)$.
\end{theorem}

In the case where the limiting load $f$ is identically zero, 
we have that
$\min_{\mathcal A_{iso}}\mathscr F_T^{lim}
=
\min_{\mathcal A_{iso}}\mathscr E_T^{lim}$
and the minimisers of 
$\mathscr E_{T}^{lim}$
are given by the following lemma.

\begin{lemma}
\label{MinPb}
We have that
\begin{equation}
\label{min_twisted}
\min_{\mathcal A_{iso}}\mathscr E_T^{lim}
\,=\,
\mathscr E_T^{lim}(y_T)
\,=\,
\frac{\mu\,\delta_0^2}{\pi^4}\left[12
\left(\frac{1+2\gamma}{1+\gamma}\right)
+\frac{\pi^4-4\pi^2-48}2\right]|\omega|,
\end{equation}
where $y_T\in\mathcal A_{iso}$
is such that
\begin{equation*}
\mbox{either}\qquad
{\rm A}_{y_T}
\equiv
{\rm diag}\left(-\frac{12\delta_0}{\pi^2(1+\gamma)},0\right)
\qquad\mbox{or}\qquad
{\rm A}_{y_T}
\equiv
{\rm diag}\left(0,\frac{12\delta_0}{\pi^2(1+\gamma)}\right).
\end{equation*}
\end{lemma}

\begin{proof}
Clearly, a deformation $y\in\mathcal A_{iso}$ which
minimises the integrand of $\mathscr E_{T}^{lim}(y)$
pointwise is a minimiser of $\mathscr E_{T}^{lim}$ over the 
class $\mathcal A_{iso}$.
Seeking for such a minimiser and since $\det{\rm A}_y(x')=0$
a.e. in $\omega$ whenever $y\in\mathcal A_{iso}$, 
we consider the problem
\begin{equation}
\label{problemi_minimo}
\min_{A\in\Sym(2):\det A=0}
\left\{
\big|A-
{\rm diag}(\alpha,-\alpha)
\big|^2
+\gamma\,
\tr^2A
\right\}
\,=\,
\min_{\xi\eta=\zeta^2}
\Big\{
(1+\gamma)(\xi+\eta)^2
+2\,\alpha\,(\alpha-\xi+\eta)
\Big\},
\end{equation}
where we have set $\alpha:=-12\delta_0/\pi^2$ and used the notation 
$\left(
\begin{array}{cc}
\xi & \zeta \\ 
\zeta & \eta\end{array}
\right)
$
to represent an arbitrary matrix $A\in{\rm Sym}(2)$.
Setting
\begin{equation*}
f(\xi,\eta)
\,:=\,
(1+\gamma)(\xi+\eta)^2
+2\,\alpha\,(\alpha-\xi+\eta)
\qquad\quad\mbox{and}\quad\qquad
g_\zeta(\xi,\eta)
\,:=\,
\xi\eta-\zeta^2,
\end{equation*}
we have that the previous minimisation problem 
can be rewritten as
$\min_{\zeta\in\R}\min_{g_\zeta(\xi,\eta)=0}f(\xi,\eta)$.
Using the method of Lagrange multipliers it is then easy to check that
\begin{equation*}
\min_{\zeta\in\R}\min_{g_\zeta(\xi,\eta)=0}f(\xi,\eta)
\,=\,
\min_{\zeta\in\R}f(\xi_\zeta^+,\eta_\zeta^+)
\,=\,
\min_{\zeta\in\R}f(\xi_\zeta^-,\eta_\zeta^-)
\,=\,
f(\xi_0^+,\eta_0^+)
\,=\,
f(\xi_0^-,\eta_0^-)
\,=\,
\alpha^2\left(\frac{1+2\gamma}{1+\gamma}\right),
\end{equation*} 
where
\begin{equation*}
\xi_\zeta^{\pm}:=
\frac12\left(\frac{\alpha}{1+\gamma}\right)
\pm
\frac12\sqrt{
\left(\frac{\alpha}{1+\gamma}\right)^2+4\zeta^2
},
\qquad\quad
\eta_\zeta^{\pm}:=
-\frac12\left(\frac{\alpha}{1+\gamma}\right)
\pm
\frac12\sqrt{
\left(\frac{\alpha}{1+\gamma}\right)^2+4\zeta^2.
}
\end{equation*}
Correspondingly, we have that the solutions
to the minimum problem on the left hand side in 
\eqref{problemi_minimo} are
\begin{equation*}
A^+
:=\,
{\rm diag}\left(\xi_0^+,\eta_0^+\right)
=
{\rm diag}\left(\frac{\alpha}{1+\gamma},0\right)
\qquad\mbox{and}\qquad
A^-
:=\,
{\rm diag}\left(\xi_0^-,\eta_0^-\right)
=
{\rm diag}\left(0,-\frac{\alpha}{1+\gamma}\right).
\end{equation*}
Now, since there exists $y\in\mathcal A_{iso}$ such that
${\rm A}_y\equiv A^+$ or ${\rm A}_y\equiv A^-$
(this corresponds to $y(\omega)$ being locally isometric
to a cylinder), we have obtained that
\begin{equation*}
\min_{y\in\mathcal A_{iso}}
\int\limits_{\omega}
\left\{\Big|{\rm A}_y(x')-\frac{12\delta_0}{\pi^2}
{\rm diag}\big(-1,1\big)\Big|^2
+\gamma\,
{\rm H}_y^2(x')\right\}\dd x'
\,=\,
|\omega|
\min_{A\in\Sym(2):\det A=0}
\left\{
\big|A-
{\rm diag}(\alpha,-\alpha)
\big|^2
+\gamma\,
\tr^2A
\right\},
\end{equation*}
and in turn that
\begin{equation*}
\min_{\mathcal A_{iso}}\mathscr E_T^{lim}
\,=\,
\frac{\mu}{12}\,|\omega|\,f(\xi_0^+,\eta_0^+)
+
\frac{\mu\,\delta_0^2}{\pi^4}
\left(\frac{\pi^4-4\pi^2-48}2\right)|\omega|\\
\,=\,
\frac{\mu}{12}\,|\omega|\,
\alpha^2\left(\frac{1+2\gamma}{1+\gamma}\right)
+
\frac{\mu\,\delta_0^2}{\pi^4}
\left(\frac{\pi^4-4\pi^2-48}2\right)|\omega|.
\end{equation*}
Substituting in the last expression the definition
of $\alpha$ gives \eqref{min_twisted}.

\end{proof}

Similarly to formula \eqref{phys_stat_splay},
the limiting plate theory for the twisted case 
can be expressed in terms of physical parameters by
\begin{multline}
\label{phys_stat_twist}
\hat{\mathscr F}_T^{h_0}(v_{h_0})
\,\cong\,
\frac{\mu\,h_0^3}{12}\int\limits_{\omega}
\left\{\Big|{\rm A}_y(x')-\frac{12\delta_0}{\pi^2}
{\rm diag}\big(-1,1\big)\Big|^2
+\gamma\,
{\rm H}_y^2(x')\right\}\dd x'\\
+
\frac{\mu\,h_0^3\delta_0^2}{\pi^4}
\Big(\frac{\pi^4-4\pi^2-48}2\Big)|\omega|
-h_0^3\int\limits_{\omega} f^{lim}(x')\cdot y(x')\,\dd x',
\end{multline}
for a given small thickness $h_0$,
where the approximate identity holds 
modulo terms of order higher than $3$ in $h_0$,
and where 
$y\in\mathcal A_{iso}$ and
$v_{h_0}$ are a minimiser of the 2D model 
\eqref{2D_twist_model}--\eqref{2D_total_energy_T} 
and a low-energy (physical) deformation,
respectively.

\smallskip
To put into perspective the two plate models
which we have derived for splay-bend and twisted nematic elastomer 
thin sheets, we conclude this section with a comparison with 
the case where a limiting plate model originates from
a three-dimensional spontaneous strain distribution
which is simpler, i.e., quadratic in the thickness variable. 
We see that, as expected, when the spontaneous strains
are kinematically compatible, 
the limiting two-dimensional stored energy 
functional is minimised at the value zero.   

\begin{remark}[The linear/quadratic case]
\upshape

Consider a system in the (physical) reference configuration $\Om_h$
endowed with a stored energy density $w^h$ of the form \eqref{prot_energy},
with the spontaneous strain distribution given by
\begin{equation}
\label{derivazione_quadrat}
\bar c_h(z_3)=
{\rm I}
+\delta_0\,z_3\,P+\eta_0\,z_3^2\,R,
\end{equation}
for some constant and dimensionless 
symmetric matrices $P$ and $R$. 
Moreover, $\delta_0$ and $\eta_0$ are 
real constants whose dimensions are inverse length and 
square of inverse length, respectively.
Let us denote by $\overline C_h(x_3)$ 
the rescaled spontaneous strain 
\begin{equation*}
\overline C_h(x_3):=\bar c_h(hx_3)
=
{\rm I}
+\delta_0\,hx_3\,P+\eta_0\,(hx_3)^2\,R,
\end{equation*}
and by $(x_3,F)\mapsto W^h(x_3,F)$ 
the corresponding stored energy density.  
Defining $\overline U_h:=\sqrt{\overline C_h}$, we have
$
W^h(x_3,F)=W_0\big(F\overline U_h^{-1}(x_3)\big)
=
W_0\big(F(
{\rm I}
+h\,B^h(x_3))\big)$,
with
$W_0$ defined as in \eqref{nostra_W} and
$B^h(x_3):=(-\delta_0\,x_3\,P/2+r_h/h)$
Since 
$B^h\longrightarrow B$ in ${\rm L}^{\infty}((-1/2,1/2),\R^{3\times3})$,
with $B(x_3):=-\delta_0\,x_3\,P/2$, then
Theorems \ref{compattezza} and \ref{thm_Schmidt_Gamma}
tell us that the limiting two-dimensional plate
model is described by the energy functional
\begin{equation}
\label{funzionale_limite_lineare}
\mathcal A_{iso}\ni
y\mapsto
\frac1{24}\int\limits_{\omega}
Q_2\left(A_y(x')-\frac{\delta_0}2\check P\right){\rm d}x.
\end{equation}
Indeed, for every $G\in\R^{2\times 2}$ we have
\begin{align*}
\overline Q_2(G)
&=
2\mu
\min_{D\in\Sym(2)}\int\limits_{-1/2}^{1/2}
\!\!\left\{
\Big|D+t\,\sym G-\frac{\delta_0}2t\check P\Big|^2
+
\gamma\,
\tr^2
\left(D+t\,G-\frac{\delta_0}2t\check P\right)
\right\}{\rm d}t
\\
&=
2\mu
\min_{D\in\Sym(2)}\int\limits_{-1/2}^{1/2}
\!\!\left\{
|D|^2+\gamma\,\tr^2D+t^2\left[\,
\Big|\sym G-\frac{\delta_0}2\check P\Big|^2
+
\gamma\,
\tr^2
\Big(G-\frac{\delta_0}2\check P\Big)
\right]
\right\}{\rm d}t\\
&=
\int\limits_{-1/2}^{1/2}
t^2\,Q_2\Big(G-\frac{\delta_0}2\check P\Big){\rm d}t
=
\frac1{12}\,
Q_2\Big(G-\frac{\delta_0}2\check P\Big).
\end{align*}
Note that the coefficient multiplying the purely quadratic
term $z_3^2$ in \eqref{derivazione_quadrat} does not play any role.
Referring to Subsection \ref{comp_quadrat} for a discussion
on the kinematic compatibility of \eqref{derivazione_quadrat},
we observe that in each of the four cases where
\eqref{derivazione_quadrat} is kinematically compatible, 
listed in the mentioned subsection,
we have that $\check P$ has at least one zero eigenvalue.
Hence, we have that the functional \eqref{funzionale_limite_lineare}
can be minimised to zero by a deformation $y\in\mathcal A_{iso}$
such that $y(\omega)$ is locally isometric to a plane
when both the eigenvalues of $\hat P$ are zero, 
and such that $y(\omega)$ is locally a cylinder
in all the other cases.
\end{remark}


\section{Energy minimising shapes under zero loads}
\label{sec_4}


The aim of this section is to give an explicit representation of the
minimal energy configurations of the
nematic sheets and to gain some physical insight on their
behaviour.
To do this, we start by characterising the deformations $y$
realising the condition $A_y\equiv{\rm diag}(k,0)$,
for some constant $k\neq0$,
under the constraint of being isometries.
More explicitly, we look for a (smooth) deformation
$y:\omega\to\R^3$ such that
\begin{equation}
\label{cond_realised}
(\na'y)^T\na'y={\rm I}_2,
\qquad\qquad\qquad
{\rm A_y}=
\left(
\begin{array}{cc}
k & 0\\
0 & 0
\end{array}
\right),
\end{equation}
or, equivalently, such that
\begin{equation*}
\left(
\begin{array}{cc}
|\pa_1y|^2 & \pa_1y\cdot\pa_2y\\
\pa_1y\cdot\pa_2y & |\pa_2y|^2
\end{array}
\right)
=
\left(
\begin{array}{cc}
1 & 0\\
0 & 1
\end{array}
\right),
\qquad\qquad
\left(
\begin{array}{cc}
\pa_1 y\cdot\pa_1\nu & \pa_1 y\cdot\pa_2\nu\\
\pa_2 y\cdot\pa_1\nu & \pa_2 y\cdot\pa_2\nu
\end{array}
\right)
=
\left(
\begin{array}{cc}
k & 0\\
0 & 0
\end{array}
\right),
\end{equation*}
where $\nu:=\pa_1y\wedge\pa_2y$.
It is easy to check that deformations
$y$ satisfying these conditions are defined 
up to arbitrary translations and superposed rotations.
Hence, we will use the normalising conditions
\begin{equation}
\label{extra_cond}
y(0,0)=0,
\qquad\qquad
\na y(0,0)=(\mathsf e_1|\mathsf e_2),
\end{equation} 
to construct one specific representative.

Note that from the condition $\pa_2y\cdot\pa_1\nu=0$
and from the identity $(\pa_1\pa_2y)\cdot\nu+\pa_2y\cdot\pa_1\nu=0$,
obtained by differentiating $\pa_2y\cdot\nu=0$ with respect to $x_1$,
one gets $(\pa_1\pa_2y)\cdot\nu=0$. Moreover,
by differentiating the conditions $|\pa_1y|^2=1$ and $|\pa_2y|^2=1$
with respect to $x_2$ and $x_1$, respectively, we obtain
\[
(\pa_1\pa_2y)\cdot\pa_1y=(\pa_1\pa_2y)\cdot\pa_2y=0.
\]  
Hence, we have that $\pa_1\pa_2y=0$ in $\omega$.
Similarly, using the condition $\pa_2y\cdot\pa_2\nu=0$ 
and suitably differentiating the identities $\pa_2y\cdot\nu=0$,
$|\pa_2y|^2=1$, and $\pa_1y\cdot\pa_2y=0$, 
one gets that $\pa_2\pa_2y=0$ in $\omega$. 
This fact, coupled with the information that the mixed derivatives of $y$
vanish, says that $y$ must be of the form
\begin{equation*}
y(x_1,x_2)
\,=\,
x_2\,\mathsf c+w(x_1),
\end{equation*}
for some $\mathsf c\in\R^3$ and some smooth
$w:\R\to\R^3$
such that $|\mathsf c|=|\dot w(x_1)|=1$ and $\mathsf c\cdot\dot w(x_1)=0$
for every $x_1$,
where we use the notation $\dot w=\pa_1w$.
Observe that
$\na'y=(\dot w\,|\,\mathsf c)$, 
that $\na'\nu=(\dot\nu\,|\,0)$, with
\[
\nu
\,:=\,
\pa_1y\wedge\pa_2y
\,=\,
\left(
\begin{array}{c}
\mathsf c_3\dot w_2-\mathsf c_2\dot w_3\\
\mathsf c_1\dot w_3-\mathsf c_3\dot w_1\\
\mathsf c_2\dot w_1-\mathsf c_1\dot w_2
\end{array}
\right),
\]
and that condition $\pa_1 y\cdot\pa_2\nu=0$ in now
automatically satisfied.
Note also that we have not exploited yet the information
that $\pa_1 y\cdot\pa_1\nu=\dot w\cdot\dot\nu=k$, 
which is going to determine the explicit expression of $w$.  
More precisely, the function $w$ has to satisfy the
following system of equations:
\[
\begin{cases}
|\dot w|^2=1,\\
\dot w\cdot\mathsf c=0,\\
\dot w\cdot\dot\nu=k.
\end{cases}
\]
To proceed, we set $f:=\dot w$ and 
choose $\mathsf c=\mathsf e_2$, 
so that the above system reduces to 
\[
\begin{cases}
f_1^2+f_3^2=1,\\
f_2=0,\\
-f_1\dot f_3+\dot f_1f_3=k.
\end{cases}
\]
Setting 
$
f_1(x_1)=\cos\big(\theta(x_1)\big)
$
and
$
f_3(x_1)=\sin\big(\theta(x_1)\big),
$
we have that the first equation is satisfied, 
while the third equation reduces to $\dot\theta(x_1)=-k$,
which yields $\theta(x_1)=-kx_1+\xi$, for some $\xi\in\R$.
In the end, we have obtained that
\[
f
\,=\,
\left(
\begin{array}{c}
f_1\\
f_2\\
f_3
\end{array}
\right)
\,=\,
\left(
\begin{array}{c}
\cos(-kx_1+\xi)\\
0\\
\sin(-kx_1+\xi)\\
\end{array}
\right)
\qquad\Rightarrow\qquad
w
\,=\,
\int^{x_1}f
\,=\,
\left(
\begin{array}{c}
-\frac1k\sin(-kx_1+\xi)+\mathsf m_1\\
\frac1k\cos(-kx_1+\xi)+\mathsf m_2\\
\mathsf m_3
\end{array}
\right)\!,
\]
for some constant $\xi\in\R$, $\mathsf m\in\R^3$.
All in all, we have that if $y:\omega\to\R^3$ is a smooth deformation 
satisfying \eqref{cond_realised}, then $\pa_2y=\mathsf c$
for some constant $\mathsf c\in\R^3$ of unit length.
Under the normalising assumption that $\mathsf c=\mathsf e_2$,
the deformation $y$ has the following expression
\begin{equation}
\label{espressione_def_min}
y(x_1,x_2)
\,=\,
x_2\,\mathsf e_2+
\left(
\begin{array}{c}
-\frac1k\sin(-kx_1+\xi)\\
0\\
\frac1k\cos(-kx_1+\xi)
\end{array}
\right)
+
\mathsf m,
\end{equation}
for some constants $\xi\in\R$ and $\mathsf m\in\R^3$.
We can now choose $\xi=0$ and $m=(0,0,-1/k)^T$, so that
$\pa_1y(0,0)=\mathsf e_1$ and $y(0,0)=0$.
Summarizing, the deformation
\begin{equation}
\label{def_min_uno}
y(x_1,x_2)
\,=\,
\left(
\frac1k\sin(kx_1),
x_2,
\frac1k\big(\cos(kx_1)-1\big)
\right)^T
\end{equation}
fulfills condition \eqref{cond_realised}
and the extra conditions \eqref{extra_cond}.

If we now look for some isometric deformation
$\tilde y$ realising the condition
$A_y\equiv{\rm diag}(0,k)$,
for some constant $k\neq0$,
namely, such that
\begin{equation}
\label{cond_realised_tilde}
(\na'\tilde y)^T\na'\tilde y={\rm I}_2,
\qquad\qquad\qquad
{\rm A_{\tilde y}}=
\left(
\begin{array}{cc}
0 & 0\\
0 & k
\end{array}
\right)\!,
\end{equation}
we can proceed similarly to the above
and check that it must be of the form 
$
\tilde y(x_1,x_2)
=
x_1\,\tilde{\mathsf c}+\tilde w(x_2),
$
for some $\tilde{\mathsf c}\in\R^3$ and some smooth
$\tilde w:\R\to\R^3$
such that $|\tilde{\mathsf c}|
=|\pa_2\tilde w(x_2)|=1$ and 
$\tilde{\mathsf c}\cdot\pa_2\tilde w(x_2)=0$
for every $x_2$.
Choosing $\tilde{\mathsf c}=\mathsf e_1$,
we easily arrive to the expression
\begin{equation*}
\tilde y(x_1,x_2)
\,=\,
x_1\,\mathsf e_1+
\left(
\begin{array}{c}
0\\
\frac1k\sin(kx_2+\xi)\\
\frac1k\cos(kx_2+\xi)
\end{array}
\right)
+
\tilde{\mathsf m}.
\end{equation*}
Choosing $\xi=0$ and $\tilde{\mathsf m}=(0,0,-1/k)^T$,
we obtain the deformation
\begin{equation}
\label{def_min_due}
\tilde y(x_1,x_2)
\,=\,
\left(x_1,
\frac1k\sin(kx_2),
\frac1k\big(\cos(kx_2)-1\big)
\right)^T\!\!\!,
\end{equation}
fulfilling  conditions
\eqref{cond_realised_tilde}
and the same extra conditions as $y$ 
in \eqref{extra_cond}. 


\smallskip
The spontaneous curvature exhibited by  minimal energy configurations in 
twisted nematic elastomer sheets cannot be read off directly from the target curvature tensor. This is because the two-dimensional bending energy \eqref{2D_twist_model}
cannot be minimised by minimising the integrand to zero, due to a geometric obstruction (there is no isometry of the plane with non-vanishing Gaussian curvature). 
This curvature is instead obtained by solving a minimisation problem, 
as shown in Lemma \ref{MinPb}.
This lemma, coupled with the above discussion,
says that the deformation $y$
defined as in \eqref{def_min_uno} with
$k=-\frac{12\delta_0}{\pi^2(1+\gamma)}$
(a portion of a cylinder
with axis parallel to the image through $y$
of the line spanned by $\mathsf e_2$,
and with radius $\bar\rho:=\pi^2(1+\gamma)/(12\delta_0)$)
and the deformation $\tilde y$
defined as in \eqref{def_min_due} with
$k=\frac{12\delta_0}{\pi^2(1+\gamma)}$
(in this case,
a portion of a cylinder
with axis parallel to the image through $\tilde y$
of the line spanned by $\mathsf e_1$,
and with the same radius $\bar\rho$)
both realise the minimum for the 2D twisted energy functional.
Nematic sheets with twisted texture are therefore \emph{bistable} under zero loads, see Figure \ref{fig:bistable}.

\begin{figure}[htbp]
\begin{center}
\includegraphics[width=7cm]{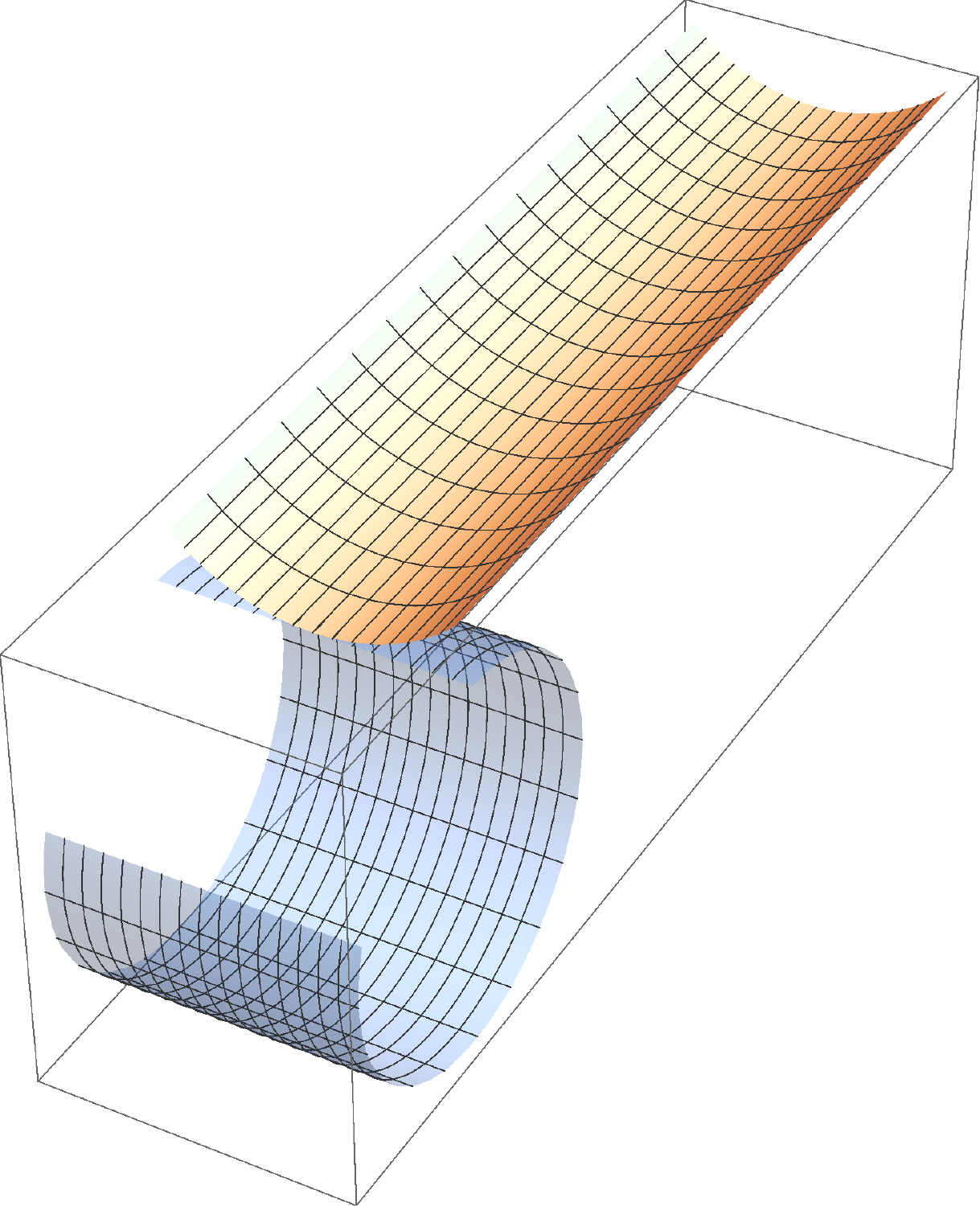}
\end{center}
\caption{Plot of minimal energy
configurations for the free-energy functional
$\mathscr E^{lim}_T$ defined in 
\eqref{2D_twist_model},
arising from a twisted texture of the nematic director.
Under zero loads, the system is stable in each of the two configurations, hence \emph{bistable}.} 
\label{fig:bistable}
\end{figure}

In the case of splay-bend textures,
the curvature  giving minimal energy can be predicted by simply 
reading it off from the target curvature tensor of the two-dimensional model \eqref{phys_stat_splay}
and therefore only deformations of type
\eqref{def_min_uno} with $k=12\delta_0/\pi^2$
(a portion of a cylinder
with axis parallel to the image through $y$
of the line spanned by $\mathsf e_2$,
and with radius $\pi^2/(12\delta_0)$)
are minimal energy states.
This means that Gaussian curvature is suppressed in the splay-bend 
as well as in the twisted case, 
in the sense that
the configurations exhibited by  elastomer thin sheets in the absence of applied loads will be portions of cylindrical surfaces (with zero Gaussian curvature, as 
predicted in  \cite{War_Mod_Cor_1,War_Mod_Cor_2} and observed experimentally in  \cite{hybrid,Sawa2011,Urayama2013}).
In both cases, these configurations carry non-zero residual stresses. In the twisted case, there will be also non-zero residual internal bending moments, due to the additional \emph{frustration} caused by the non-attainability of the target curvature. In the splay-bend case the target curvature is attained, the bending energy is minimised to zero, and no residual moments arise.


\smallskip
It is worth comparing the case of twist and splay-bend textures 
with a different scenario, in which
the nematic director is kept constant along the 
thickness of the thin sheet,
whereas the spontaneous strain \eqref{nem_tens}
varies along the thickness through the magnitude parameter $a$.
To the best of our knowledge, a system with these features has not yet been synthesized in a laboratory. 
At least in principle, this should be possible by realising a film 
with uniform alignment of the director perpendicular to the 
mid-suface (direction $\mathsf e_3$), 
and variable degree of order along the thickness ($x_3$ coordinate).

Using the notation of Subsection \ref{subsec:three_dim_model},
let us suppose that the nematic director
$n^h(z_3)$ is now constant, and equal to some $\mathsf n\in\Sph^2$,
and that the (constant) parameter
$a_h$ in \eqref{a_h} is now given by
\begin{equation*}
\bar a_h(z_3)
\,:=\,
1+\frac{\alpha_0}{h_0}z_3,
\qquad\qquad
z_3\in(-h/2,h/2).
\end{equation*}
The (physical) spontaneous strain
of this system is therefore defined as
\begin{equation*}
\bar c_h(z_3)
\,:=\,
\bar a_h^{2/3}(z_3)\,\mathsf n\otimes\mathsf n
+\bar a_h^{-1/3}(z_3)\big({\rm I}-\mathsf n\otimes\mathsf n\big).
\end{equation*}
Modelling the system using again the prototypical
energy density \eqref{prot_energy},
as in Subsection \ref{3D_rescaled}
we can define the rescaled energy densities $W_h(x_3,\cdot)$,
$x_3\in(-1/2,1/2)$,
characterised by the (rescaled) spontaneous strains
\begin{align*}
\overline C_h(x_3)
\,:=\,
\bar c_h(hx_3)
&
\,=\,\Big(1+\frac{\alpha_0}{h_0}h\,x_3\Big)^{2/3}
\!\mathsf n\otimes\mathsf n
+
\Big(1+\frac{\alpha_0}{h_0}h\,x_3\Big)^{-1/3}
\big({\rm I}-\mathsf n\otimes\mathsf n\big)\nonumber\\
&
\,=\,{\rm I}-2\,h\,B(x_3)+R^h(x_3),
\qquad\quad
B(x_3)
\,:=\,
\frac{x_3}2\frac{\alpha_0}{h_0}
\left(\frac{\rm I}3-\mathsf n\otimes\mathsf n\right),
\end{align*}
where $\|R^h\|_{\infty}=o(h)$.
Proceeding as in Subsection \ref{nostra_derivazione},
we obtain a limit 2D model 
whose free-energy functional is given by
\begin{equation}
\label{funzionale_Sharon}
\mathscr E^{lim}(y)
\,=\,
\frac12
\int\limits_{\omega}
\overline Q_2(A_y(x'))
{\rm d}x'
\,=\,
\frac1{24}
\int\limits_{\omega}
Q_2\left(A_y(x')-\check M\right)
{\rm d}x',
\end{equation}
for every $y\in\mathcal A_{iso}$.
Here, the $2{\times}2$ symmetric matrix $\check M$
is given by the formula
\begin{equation*}
\check M
\,=\,
\frac12\frac{\alpha_0}{h_0}
\left[(\mathsf n\otimes\mathsf n)^{\check{}}-
\frac{{\rm I}_2}3\right],
\end{equation*}
and 
$(\mathsf n\otimes\mathsf n)^{\check{}}$
is the $2{\times}2$ upper left part of 
$\mathsf n\otimes\mathsf n$.
In the case where  
$\mathsf n=\mathsf e_3$, 
the spontaneous curvature tensor $\check M$
reduces to 
\begin{equation}
\label{funzionale_Sharon_2}
\check M
\,=\,
\left(
\begin{array}{ccc}
m_0 & 0\\
0 & m_0
\end{array}
\right),
\qquad\qquad
m_0\,:=\,
-\frac{\alpha_0}{6h_0}.
\end{equation}
Note that the Gaussian curvature
associated with $\check M$ is positive.
However, as for the twisted case, 
where the spontaneous Gaussian curvature is negative,
the observable minimal energy
configurations will always exhibit zero Gaussian curvature
(see Lemma \ref{MinPb}),
because of the isometry constraint
they are subjected to.
More precisely, 
some calculations show that every isometric deformation
$\bar y\in\mathcal A_{iso}$
such that $A_{\bar y}\equiv\bar A_+(s)$ 
or $A_{\bar y}\equiv\bar A_-(s)$ for some
$s\in[-\bar k/2,\bar k/2]$,
with
\begin{equation}
\label{bar_k}
\bar A_\pm(s)
\,:=\,
\left(
\begin{array}{ccc}
\frac{\bar k\pm\sqrt{\bar k^2-4s^2}}2 & s\\
s & \frac{\bar k\mp\sqrt{\bar k^2-4s^2}}2
\end{array}
\right)\!,
\qquad\qquad\quad
\bar k
\,:=\,
m_0\Big(\frac{1+2\gamma}{1+\gamma}\Big),
\end{equation}
is such that
\[
\mathscr E^{lim}(\bar y)
\,=\,
\min \mathscr E^{lim}
\,=\,
\frac{\mu}{12}\,m_0^2\,\Big(\frac{1+2\gamma}{1+\gamma}\Big)
|\omega|,
\]
where the expression of $\gamma$ in terms
of the 3D parameters is given in formula 
\eqref{def:lambda}.
Note that
$\bar A_-(-\bar k/2)=\bar A_+(-\bar k/2)$ and 
$\bar A_-(\bar k/2)=\bar A_+(\bar k/2)$,
whereas
$\bar A_-(s_1)\neq\bar A_+(s_2)$
for every $s_1,s_2\in(-\bar k/2,\bar k/2)$.
Note also that the eigenvalues of $\bar A_+(s)$ and $\bar A_-(s)$ are
always $\bar k$ and $0$ for every $s\in[-\bar k/2,\bar k/2]$,
and that 
\begin{equation}
\label{bar_A_+_e_-}
\bar A_+(s)
\,=\,
R_+(s)
\left(
\begin{array}{ccc}
\bar k & 0\\
0 & 0
\end{array}
\right)
R_+(s)^T,
\qquad\qquad
\bar A_-(s)
\,=\,
R_-(s)
\left(
\begin{array}{ccc}
0 & 0\\
0 & \bar k
\end{array}
\right)
R_-(s)^T,
\end{equation}
where the columns of the rotation matrices 
$R_+(s)$ and $R_-(s)$ are the eigenvectors corresponding
to $\bar k$ and $0$ and to $0$ and $\bar k$,
respectively. The explicit expressions of 
$R_+(s)$ and $R_-(s)$ are the following:
\begin{equation}
\label{rotazione_1}
R_\pm(s)
\,=\,
\sqrt{\frac{\bar k+\sqrt{\bar k^2-4s^2}}{2\bar k}}
\left(
\begin{array}{ccc}
1 & \mp\frac{2s}{\bar k+\sqrt{\bar k^2-4s^2}}\\
\pm\frac{2s}{\bar k+\sqrt{\bar k^2-4s^2}} & 1
\end{array}
\right)\!,
\end{equation}
In particular, for $\bar A_+$, 
the directions corresponding to the eigenvalue $0$ 
are given, respectively, 
by the vector $(1,1)$ in the case $s=-\bar k/2$,
by $(0,1)$ in the case $s=0$,
and by $(1,-1)$ in the case $s=\bar k/2$.
For $\bar A_s$, 
the directions corresponding to the eigenvalue $0$ 
are given, respectively, 
by the vector $(1,1)$ in the case $s=-\bar k/2$,
by $(1,0)$ in the case $s=0$,
and by $(1,-1)$ in the case $s=\bar k/2$.
Therefore, through the matrices $\bar A_+$ and 
$\bar A_-$ all the possible directions corresponding
to the eigenvalue $0$ (and in turn
all the corresponding orthogonal eigenspaces 
corresponding to the eigenvector $\bar k$)
are represented.
All in all, we have that
\[
\mathscr A\,:=\,
\left\{\bar A_+(s):s\in
\left[-\frac{\bar k}2,\frac{\bar k}2\right]
\right\}
\cup
\left\{\bar A_-(s):s\in
\left[-\frac{\bar k}2,\frac{\bar k}2\right]
\right\}
\,=\,
\Big\{R\,{\rm diag}(\bar k,0)\,R^T: R\in{\rm SO}(2)\Big\}
\] 

From the discussion leading to
expression \eqref{espressione_def_min}, we have that, 
given $\bar\rho>0$, 
the deformation defined as 
\[
y(x_1,x_2)
\,:=\,
\left(\bar\rho\sin\Big(\frac{x_1}{\bar\rho}\Big),
x_2,
\bar\rho\cos\Big(\frac{x_1}{\bar\rho}\Big)\right)^T
\]
is an isometry such that
$A_y\equiv{\rm diag}(1/\bar\rho,0)$.
Moreover, we have that
\begin{equation}
\label{norm}
y(0,0)\,=\,(0,0,\bar\rho)^T,
\qquad\qquad
\na y(0,0)=(\mathsf e_1|\mathsf e_2).
\end{equation}
Consider the rotation matrices 
\[
R_{\alpha}
\,:=\,
\left(
\begin{array}{ccc}
\sin\alpha & -\cos\alpha & 0\\
\cos\alpha & \sin\alpha & 0\\
0 & 0 & 1
\end{array}
\right),
\qquad\qquad
\check R_{\alpha}
\,:=\,
\left(
\begin{array}{cc}
\sin\alpha & -\cos\alpha\\
\cos\alpha & \sin\alpha
\end{array}
\right),
\]
and note that $\check R_{\alpha}$ is 
a $(\pi/2-\alpha)$-counterclockwise rotation
taking the vector 
$(\cos\alpha,\sin\alpha)^T$ into $(0,1)^T$.
Now, setting $\bar\rho=1/\bar k$, where
$\bar k$ is defined as in \eqref{bar_k}, 
we define the deformation
\begin{equation}
\label{deformazioni_minime}
y^\alpha(x')
\,:=\,
R_{\alpha}^T\,\circ\,y\,\circ\,\check R_{\alpha}(x')
\,=\, 
\left(
\begin{array}{c}
\bar\rho\sin\alpha\sin\Big(\frac{x_1\sin\alpha-x_2\cos\alpha}{\bar\rho}\Big)
+
\cos\alpha
(x_1\cos\alpha+x_2\sin\alpha)\\
-\bar\rho\cos\alpha
\sin\Big(\frac{x_1\sin\alpha-x_2\cos\alpha}{\bar\rho}\Big)
+
\sin\alpha
(x_1\cos\alpha+x_2\sin\alpha)\\
\bar\rho\cos\alpha
\cos
\Big(\frac{x_1\sin\alpha-x_2\cos\alpha}{\bar\rho}\Big)
\end{array}
\right).
\end{equation}
Simple computations show that $y^\alpha\in\mathcal A_{iso}$
and that,
setting $\nu^\alpha:=\pa_1y^\alpha\wedge\pa_2y^\alpha$,
\begin{equation*}
A_{y^\alpha}(x')
\,:=\,
(\na'y^\alpha(x'))^T\na'\nu^\alpha(x')
\,=\,
\check R_{\alpha}^T\,A_{y} (\check R_{\alpha}x')\,
\check R_{\alpha}
\,=\,
\check R_{\alpha}^T
\left(
\begin{array}{ccc}
\frac1{\bar\rho} & 0\\
0 & 0
\end{array}
\right)
\check R_{\alpha}
\,\in\,
\mathscr A.
\end{equation*}
Therefore, 
we have that
$A_{y^\alpha}\equiv\bar A_+(s)$
for some $s\in[-\bar k/2,\bar k/2]$
and $\alpha\mapsto y^{\alpha}$ 
is a continuous family of deformations 
minimising $\mathscr E^{lim}$.
Note also that $y^\alpha$ satisfies the normalizing conditions
\eqref{norm}, for every $\alpha$. 
Figure \ref{fig:sharon}
shows minimal energy deformed configurations
obtained from the family $\alpha\mapsto y^{\alpha}$.

The existence of a continuous family of
deformations with (constant) minimal energy
shows that a nematic elastomer sheet with 
constant director (perpendicular to the mid-surface)
and thickness-dependent magnitude of the spontaneous strain
(this can be realised by varying the degree of the 
nematic order along the thickness)
realises a ``zero-stiffness'' structure in the sense of 
\cite{Guest}.
These are structures that can undergo large elastic deformations
without requiring external work.
Figure \ref{fig:sharon} show that the nematic sheet 
can accomodate any level of  twisting
with negligible elastic energy in between two
extreme states 
($\alpha=0$ and $\alpha=\pi/2$). 
Of course, zero-stiffness is an idealisation and, in a real system,
effects that have not been taken into 
account in the model will lead 
to small, but non-zero loads in order to change shape.
In the example of Sharon \cite{Levin_Sharon}, 
edge effects cause energy storage which scales as $h^{7/2}$.
This is a higher scaling (with smaller stored energy in the thin film limit $h\to 0$) with respect to the bending one 
($h^3$) that our dimensionally reduced theory is designed to resolve. 
As a consequence, the observed response is much ``softer'' 
than the one expected from the bending stiffness of a sheet.

By contrast, sheets with twist texture
are ``bistable'' in the sense of \cite{Guest_BIS}: they exhibit two distinct possible stable shapes
in the absence of loads (see Figure \ref{fig:bistable}).
Splay-bend sheets have only one shape
minimising the energy under zero loads.


\begin{figure}[htbp]
\begin{center}
\includegraphics[width=10cm]{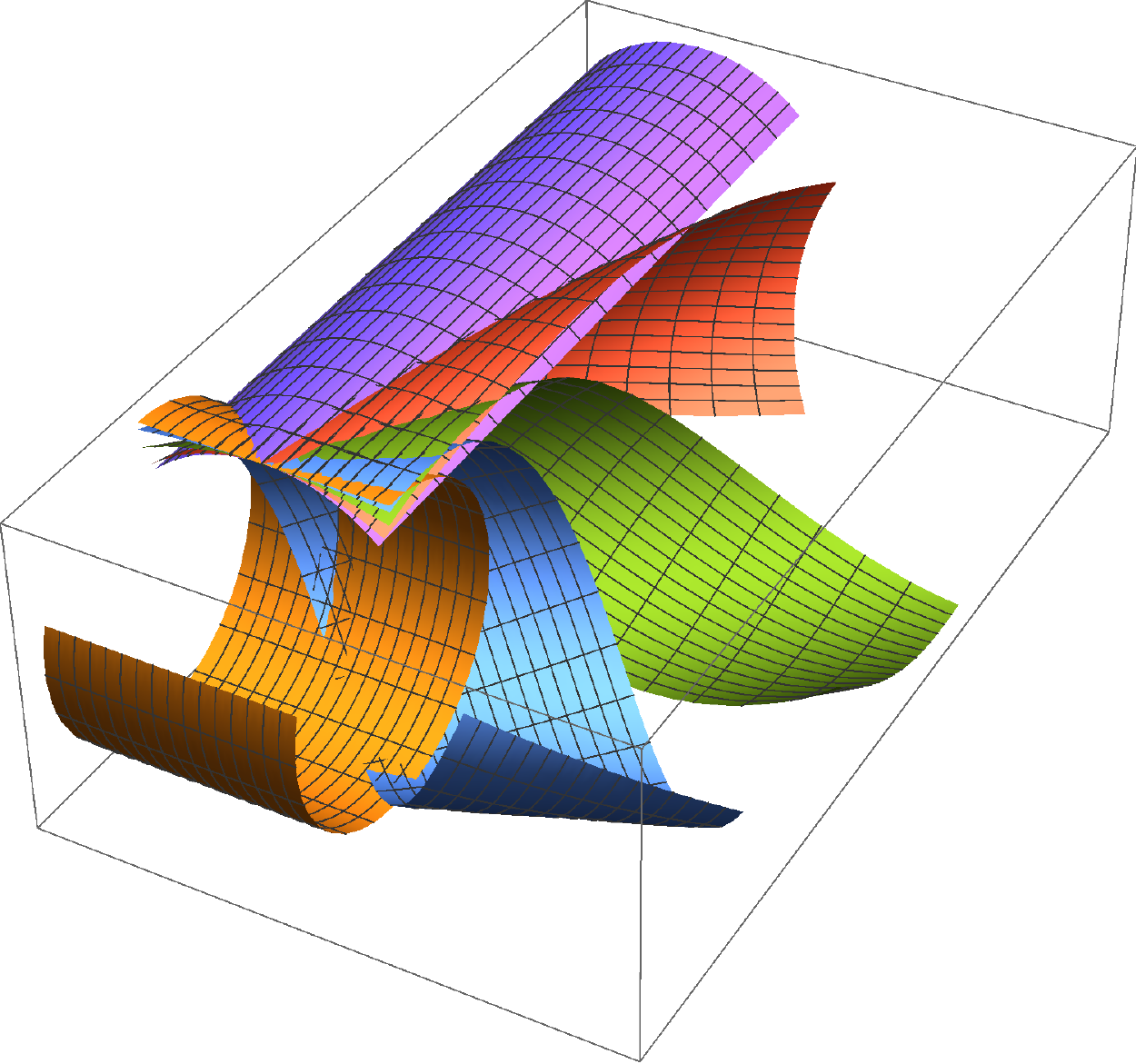}
\end{center}
\caption{
Plot of minimal energy
configurations for the free-energy functional
$\mathscr E^{lim}$ defined in 
\eqref{funzionale_Sharon}--\eqref{funzionale_Sharon_2},
arising from constant director 
$\mathsf e_3$ along the thickness of the sheet.
The configurations are elements of 
the continuous family of deformations \eqref{deformazioni_minime}.
Under zero loads, the system has the same (minimal) energy in each of these configurations, hence it realises a structure
with  \emph{zero-stiffness} to twisting.} 
\label{fig:sharon}
\end{figure}


\subsection*{Acknowledgements}
We gratefully acknowledge the support by the European Research Council through the ERC Advanced Grant 340685-MicroMotility.
We thank S. Guest, R. V. Kohn, and E. Sharon for valuable discussions.




\bibliographystyle{plain}


\end{document}